\documentclass[10pt,a4paper]{amsart}
\usepackage[utf8]{inputenc}
\usepackage[french,english]{babel}
\usepackage[T1]{fontenc}
\usepackage{amsmath}
\usepackage{amsfonts}
\usepackage{amssymb}
\usepackage{amsaddr}

\usepackage{amsthm}
\usepackage{color}
\usepackage{bbm}
\usepackage{fullpage}
\usepackage{comment}
\usepackage[hidelinks]{hyperref}
\author{Malo Jézéquel}
\address{\'Elève au DMA, ENS, 45 rue d'Ulm, 75005, Paris \\ (e-mail : malo.jezequel@ens.fr)} 

\title{Réponse linéaire et points périodiques : cas analytique.}
\newcommand{\B}{\mathcal{B}}
\renewcommand{\(}{\left(}
\renewcommand{\)}{\right)}

\newcommand{\N}{\mathbb{N}}
\newcommand{\R}{\mathbb{R}}
\newcommand{\T}{\mathbb{T}}
\newcommand{\Z}{\mathbb{Z}}
\newcommand{\C}{\mathbb{C}}
\newcommand{\A}[1]{\mathcal{A}\(#1\)}
\newcommand{\D}[2]{\mathbb{D}\left(#1,#2\right)}
\renewcommand{\L}{\mathcal{L}}
\newcommand{\tr}{\mathrm{tr}}
\newcommand{\s}[3]{\int_{#1} #2 \mathrm{d}#3}
\newcommand{\set}[1]{\left\{#1\right\}}

\newcommand{\fix}[1]{\mathrm{Fix} \( #1\)}
\newcommand{\drond}[2]{\frac{\partial #1}{\partial #2}}
\newcommand{\ddeux}[3]{\frac{\partial^2 #1}{\partial #2 \partial #3}}
\newcommand{\bul}[1]{\stackrel{\circ}{#1}}
\renewcommand{\phi}{\varphi}
\newcommand{\oh}{\hat{\otimes}}
\renewcommand{\bar}[1]{\overline{#1}}
\newtheorem{lm}{Lemme}[subsection]
\newtheorem{prop}[lm]{Proposition}
\newtheorem{df}[lm]{Définition}
\newtheorem{cor}[lm]{Corollaire}
\newtheorem{rmq}[lm]{Remarque}
\newtheorem{thm}[lm]{Théorème}
\begin{document}

\selectlanguage{english}

\begin{abstract}
We explain some points of the demonstration by Pollicott and Vytnova of a formula for linear response in terms of periodic points in the case of analytic expanding maps of the circle or Anosov diffeomorphisms of the torus. The main tool is the dynamical determinant which may be written as the Fredholm determinant of a transfer operator (or a quotient of Fredholm determinants of transfer operators).
\end{abstract}

\selectlanguage{french}

\maketitle

\section*{Introduction}

On explicite certains points peu détaillés de la démonstration d'une formule due à Pollicott et Vytnova (\cite{Poll}, sur une idée de Cvitanovic, voir \cite{Cvi}) pour la réponse linéaire, dans le cas d'une perturbation d'une application dilatante du cercle ou d'un difféomorphisme d'Anosov du tore, qui peut s'exprimer uniquement en fonction de la perturbation aux points périodiques du système initial et produit de très bonnes approximations. C'est-à-dire que si $t \mapsto f_t$ est une perturbation analytique d'une dynamique $f_0$ et $g$ une observable analytique on donne une formule pour
\[
\left. \drond{}{t} \( \s{}{g}{\mu_t} \) \right|_{t=0}
\]
où $\mu_t$ désigne la mesure physique de $f_t$. La formule dans le cas d'applications dilatantes du cercle est l'objet des corollaires \ref{derivee} et \ref{formula} et s'adapte immédiatement au cas de difféomorphismes d'Anosov du tore comme l'explique la remarque \ref{adaptation} (il s'agit du lemme 3.1 de \cite{Poll}).

L'ingrédient principal de la démonstration est le déterminant dynamique qui peut être défini formellement en dimension $1$ par
\[
z \mapsto \exp \(- \sum_{n > 0} \frac{1}{n} \( \sum_{x \in \fix{f^n}} \frac{\prod_{k=0}^{n-1} G\(f^k\(x\)\)}{\left|\(f^n\)'\(x\)-1\right|}\) z^n \)
\]
pour une dynamique $f$ et un poids $G$. Le principal apport de notre travail est l'étude de la régularité de ce déterminant lorsque l'on fait varier $f$ et $G$, justifiant ainsi certains points peu détaillés de \cite{Poll}, notamment les lemmes 2.6 et 2.13 dudit article. C'est l'objet de la proposition \ref{ppal} dans le cas d'une application dilatante du cercle et du théorème \ref{gros} dans le cas d'un difféomorphisme Anosov du cercle.

On traite dans la première partie le cas d'une application dilatante du cercle, en utilisant un déterminant dynamique dont l'usage a été introduit par Ruelle dans \cite{Ruelle1}. Le cadre géométrique particulièrement simple nous permet d'en simplifier la construction, évitant ainsi tout recours aux partitions de Markov et à la dynamique symbolique. En particulier, le déterminant ainsi construit est une fonction entière.

On traite dans la deuxième partie le cas des difféomorphismes Anosov du tore, suivant cette fois \cite{R1} et \cite{R2} pour la construction du déterminant dynamique.

Ce texte est le résultat d'un travail mené lors d'un stage de M2 de l'université d'Orsay sous la direction de Viviane Baladi. 

\tableofcontents

\section{Applications dilatantes du cercle}

La formule que l'on va donner pour la réponse linéaire se base sur l'étude d'un déterminant dynamique pour une application dilatante du cercle analytique $f$ et un poids analytique $G$, c'est-à-dire d'une fonction entière qui prolonge 
\begin{equation}\label{annonce}
z \mapsto \exp \(- \sum_{n > 0} \frac{1}{n} \( \sum_{x \in \fix{f^n}} \frac{\prod_{k=0}^{n-1} G\(f^k\(x\)\)}{\left|\(f^n\)'\(x\)-1\right|}\) z^n \)
.\end{equation}

L'étude de ce déterminant est l'objet de \S \ref{11}. On notera que dans le cas particulier que nous traitons ce déterminant peut être vu comme le déterminant de Fredholm d'un opérateur de transfert agissant sur un espace de fonctions analytiques (corollaire \ref{formule}). On évite en effet le recours à la dynamique symbolique et à la formule de Manning et la formule des résidus donne des démonstrations plus directes de plusieurs résultats (pour le cas général voir \cite{Ruelle1}).

On justifie dans \S \ref{12} que si $f_t$ est une application dilatante du cercle qui dépend de manière analytique de $t$, alors son déterminant associé au poids $\exp\(ug\)$ dépend également de manière analytique de $t$ et de $u$ (voir la proposition \ref{ppal}). La régularité du déterminant est implicitement utilisée dans \cite{Poll} et il nous semblait bon d'en donner une démonstration. 

On utilise cette étude du déterminant pour donner dans \S \ref{13} la formule pour la réponse linéaire annoncée, rappelant le résultat principal de \cite{Poll}. C'est-à-dire que si $g$ est une fonction analytique du cercle dans $\R$ et $\mu_t$ désigne la mesure invariante pour $f_t$ absolument continue devant Lebesgue, on donne dans les corollaires \ref{derivee} et \ref{formula} des formules pour
\[
\drond{}{t} \left. \( \s{S^1}{g}{\mu_t} \) \right|_{t=0}
.\]

Enfin, on donne dans \S \ref{14} une estimation de la vitesse à laquelle les coefficients du développement en série entière du déterminant tendent vers $0$ (proposition \ref{vitconv}). On donne ainsi une démonstration du lemme 2.6 de \cite {Poll}, ce qui montre que la formule énoncée dans le corollaire \ref{formula} donne lieu à de très bonnes approximations.

\subsection{Déterminant dynamique d'une application dilatante du cercle}\label{11}

Nous justifions ici l'existence du déterminant dynamique par une version simplifiée de l'argument de \cite{Ruelle1}. Cette simplification est permise par le cadre géométrique simple et assure que le déterminant dynamique est une application entière.

Commençons par fixer quelques notations. On note $\pi$ la surjection canonique de $\C$ sur $\C / \Z$ et $\psi$ le $\mathcal{C}^{\omega}$-difféomorphisme entre $\C / \Z$ et $\C^*$ obtenu par factorisation de $ x \mapsto \exp\(2i \pi x\)$. On voit $S^1 = \R / \Z$ comme une partie de $\C / \Z$. Pour tout réel $r$ strictement positif on note $\mathcal{C}_r = \set{ z \in \C / \Z : |\Im z| < r}$.

On se donne une application $f$ de $S^1$ dans $S^1$ analytique et uniformément dilatante : il existe $\lambda > 1$ tel que pour tout $x \in S^1$ on ait $f'\(x\) \geqslant \lambda$. On étend $f$ en une application holomorphe sur un voisinage $U$ de $S^1$ dans $\C / \Z$ de telle manière qu'on ait toujours $f'\(z\) \neq 0$ pour $z \in U$.

On note $F$ le relevé de $f$ qui est défini sur $\pi^{-1} U$ à valeurs dans $\C$. On pose aussi $\tilde{f} = \psi \circ f \circ \psi^{-1} $ (définie sur $\psi U$). Ce recours à différents modèles du cercle encombre légèrement les notations mais simplifie grandement un certain nombre de preuves.

Quelques résultats techniques sont nécessaires avant de pouvoir définir l'opérateur de transfert que l'on va étudier.

\begin{lm}\label{tech}
Soit $r$ un réel strictement positif tel que $\mathcal{C}_r \subseteq U$. On note $M = \underset{\substack{\Re z \in \left[ - \frac{r}{2} , 1 + \frac{r}{2} \right] , \\ \Im z \in \left[ - \frac{r}{2} , \frac{r}{2} \right]} }{\sup} |F\(z\)|$. Soit $0 < \lambda' < \lambda$. Alors pour tout $z \in \C / \Z $ tel que $|\Im z | <  \min\(\frac{r}{4}, \frac{r^2\(\lambda - \lambda'\)}{8M}\) $ on a $|\Im f\(z\)| \geqslant \lambda' |\Im z|$.
\end{lm}

\begin{proof}
Soit $x +iy$ un relevé de $z$ tel que $x \in \left[0,1\right]$ et $y = \Im z$. On a alors :
\[
F\(x+iy\) = F\(x\) + iyF'\(x\) + \sum_{n \geqslant 2} a_n \(iy\)^n
,\] 
où pour $n \geqslant 2$ on a posé $a_n = \frac{1}{2i \pi} \s{\partial \D{x}{\frac{r}{2}}}{\frac{F\(w\)}{\(w-x\)^{n+1}}}{w}$
de telle manière que
$|a_n| \leqslant \frac{2^n M}{r^n}$
et donc
\begin{align*}
\left| \sum_{n \geqslant 2} a_n \(iy\)^n \right| & \leqslant M \sum_{n \geqslant 2} \(\frac{2|y|}{r}\)^n \\
              & \leqslant M \(\frac{2|y|}{r}\)^2 \frac{1}{1- \frac{2|y|}{r}} \leqslant 8M \(\frac{|y|}{r}\)^2 \\
              & \leqslant \(\lambda - \lambda'\) |y|
,\end{align*}
et il vient alors
\begin{align*}
|\Im f\(z\)| & = |\Im F\(x+iy\)| \geqslant |y||F'\(x\)| - \left| \sum_{n \geqslant 2} a_n \(iy\)^n \right| \\
             & \geqslant \lambda|y| - \(\lambda -\lambda'\) |y| \geqslant \lambda' |\Im z|
.\end{align*}
\end{proof}

Remarquons que $f$ induit un revêtement du cercle par lui-même (c'est un homéomorphisme local propre). On note $D$ son nombre de feuillets (qui est aussi son degré et est fini). Pour tout réel strictement positif $r$ on note $\mathbb{A}_r = \psi \mathcal{C}_r = \set{z \in \C : e^{-2 \pi r} < |z| < e^{2 \pi r}}$.

\begin{lm}\label{antec}
Soit $0 < r_0 < \min\(\frac{r}{4}, \frac{r^2\(\lambda - 1\)}{8M}\)$ avec les notations du lemme précédent. Alors $\mathcal{C}_{r_0} \subseteq U$ et tout élément de $\mathcal{C}_{r_0}$ a exactement $D$ antécédents par $f$ dans $\mathcal{C}_{r_0}$.
\end{lm}

\begin{proof}
On a $r < r_0$ et donc $\mathcal{C}_{r_0} \subseteq U$. Donnons nous $\lambda'$ tel que $1 < \lambda' < \lambda$ et $ r_0 < \frac{r^2\(\lambda - \lambda'\)}{8M}$. Maintenant si $w \in \mathcal{C}_{r_0}$ le nombre d'antécédents de $w$ par $f$ dans $\mathcal{C}_{r_0}$ est le même que le nombre d'antécédents de $\psi\(w\)$ par $\tilde{f}$ dans $\mathbb{A}_{r_0}$ c'est-à-dire :
\[
\frac{1}{2 i \pi} \s{\partial \mathbb{A}_{r_0}}{\frac{\tilde{f}'\(z\)}{\tilde{f} \(z\) - \psi\(w\)}}{z}
.\]
Or cette fonction est continue en $w$ puisqu'on a la domination (avec le lemme précédent) :
\[
\left| \frac{\tilde{f}'\(z\)}{\tilde{f} \(z\) - \psi\(w\)} \right| \leqslant \frac{\sup_{\partial \mathbb{A}_{r_0}} |\tilde{f}'|}{ 2 \(\sinh\(\lambda' r_0\) - \sinh\(r_0\)\)}
\]
pour tout $z \in \partial \mathbb{A}_{r_0}$ et tout $w \in \mathcal{C}_{r_0}$.
Par connexité, les éléments de $\mathcal{C}_{r_0}$ ont donc tous le même nombre d'antécédents dans $\mathcal{C}_{r_0}$. Or un point du cercle a $D$ antécédents sur le cercle et aucun autre dans $\mathcal{C}_{r_0}$ (par le lemme précédent).
\end{proof}

\begin{lm}\label{rev}
Avec les notations du lemme précédent, $f$ induit un revêtement holomorphe à $D$ feuillets de $\mathcal{C}_{r_0} \cap f^{-1} \mathcal{C}_{r_0}$ sur $\mathcal{C}_{r_0}$. En particulier, $\mathcal{C}_{r_0} \cap f^{-1} \mathcal{C}_{r_0}$  est connexe (par arcs). 
\end{lm}

\begin{proof}
Un homéomorphisme local à fibres finies de cardinal constant est un revêtement. Tout élément de $\mathcal{C}_{r_0}$ est dans la même composante connexe (par arcs) que le cercle et donc par relèvement, en remarquant que $f^{-1}S^1 = S^1$, tout élément de $\mathcal{C}_{r_0} \cap f^{-1} \mathcal{C}_{r_0}$ aussi.
\end{proof}

Dans toute la suite on fixe un tel $r_0$ et $R < r_0$. On se donne une fonction holomorphe $G$ d'un voisinage de  $\bar{\mathcal{C}_R}$ dans $\C$ et on note $\widetilde{G} = G \circ\psi^{-1}$. Nous pouvons maintenant définir l'opérateur de transfert qui nous intéresse ici.

\begin{df}
Si $U$ est un ouvert de $\C$ ou de $\C / \Z$, on note $\A{U}$ l'espace des fonctions continues bornées sur $\bar{U}$ à valeurs dans $\C$ dont la restriction à $U$ est holomorphe. On munit $\A{U}$ de la norme du sup.
\end{df}

\begin{df}
Si $\phi \in \A{\mathcal{C}_R}$ et $z \in \bar{\mathcal{C}_R}$ on pose 
\[
\L \phi \(z\) = \sum_{x \in \mathcal{C}_R : f\(x\) = z} G\(x\) \phi\(x\)
.\]
De même si $\phi \in \A{\mathbb{A}_R}$ et $z \in \bar{\mathbb{A}_R}$ on pose 
\[
\tilde{\L} \phi \(z\) = \sum_{x \in \mathbb{A}_R : \tilde{f} \(x\) = z} \widetilde{G} \(x\) \phi\(x\)
,\]
où $\tilde{f} = \psi \circ f \circ \psi^{-1}$ et $\widetilde{G}= G \circ \psi^{-1}$ comme ci-dessus.
\end{df}

\begin{rmq}
Les antécédents qui apparaissent dans la première somme sont en fait dans $\bar{\mathcal{C}_{\frac{R}{\lambda'}}}$ (et ceux de la deuxième dans $\bar{\mathbb{A}_{\frac{R}{\lambda'}}}$) où $\lambda'$ est défini comme dans la démonstration du lemme \ref{antec}.
\end{rmq}

\begin{prop}\label{defin}
$\L$ et $\tilde{\L}$ définissent des opérateurs bornés respectivement de $\A{\mathcal{C}_R}$ dans lui-même et de $\A{\mathbb{A}_R}$ dans lui-même. De plus, l'application $\phi \mapsto \phi \circ \psi^{-1}$ définit une isométrie entre les Banach $\A{\mathcal{C}_R}$ et $\A{\mathbb{A}_R}$ qui conjugue $\L$ et $\tilde{\L}$.
\end{prop}

\begin{proof}
Ce qui n'est pas trivial est une conséquence immédiate du lemme \ref{rev}.
\end{proof}

Si $n$ est un entier naturel et $x \in \bar{\mathcal{C}_R}$, on note $A_n\(x\)$ l'ensemble des antécédents $y$ de $x$ par $f^n$ tels que $y,\dots, f^{n-1}\(y\)$ sont dans $\bar{\mathcal{C}_R}$. Si $y \in \cap_{k=0}^{n-1} f^{-k} V $, où $V$ est le domaine de $G$, on note $G_n\(y\) = \prod_{k = 0}^{n-1} G \(f^k \(y\)\)$.

Si $\phi \in \A{\mathcal{C}_R}$ et $z \in \bar{\mathcal{C}_R}$ on a alors 
\[
\L^n \phi \(x\) = \sum_{y \in A_n\(x\) } G_n\(y\) \phi\(y\)
.\]

De même si $\phi \in \A{\mathbb{A}_R}$ et $z \in \bar{\mathbb{A}_R}$ on a
\[
\tilde{\L}^n \phi \(x\) = \sum_{y \in \tilde{A_n}\(x\)} \widetilde{G}_n \(y\) \phi\(y\)
\]
où on a posé $\tilde{G}_n = G_n \circ \psi^{-1}$ et $\tilde{A_n}\(x\) = \psi\(A_n\(\psi^{-1}\(x\)\)\)$  pour $x \in \mathbb{A}_R$.

\begin{rmq}\label{cont}
Si $x \in \mathcal{C}_R$, les éléments de $A_n\(x\)$ sont dans $\mathcal{C}_{\frac{R}{\lambda'^n}}$ où $\lambda'$ est défini comme dans la démonstration du lemme \ref{antec}.
\end{rmq}

Les résultats qui suivent constituent l'essentiel du contenu de cette partie : ils justifient l'existence du déterminant dynamique et en donnent une expression en fonction de l'opérateur de transfert $\L$. Pour des raisons techniques, on traite l'essentiel de la démonstration en utilisant plutôt l'opérateur $\tilde{\L}$. Nous allons travailler avec des opérateurs nucléaires au sens de Grothendieck (voir \cite{Groth}). Notons que les opérateurs nucléaires que nous étudions ici sont systématiquement $\frac{2}{3}$-sommables et ont donc une trace et un déterminant bien définis (voir le corollaire 4 page 18 de la deuxième partie de \cite{Groth}).

\begin{lm}\label{calc}
L'opérateur $\tilde{\L}$ est $\frac{2}{3}$-sommable et pour tout entier naturel non-nul $n$ on a 
\[
\tr\(\tilde{\L}^n\) = \sum_{x \in \fix{\tilde{f}^n}} \frac{\(\tilde{f}^n\)'\(x\) \widetilde{G}_n\(x\)}{\(\tilde{f}^n\)'\(x\) - 1}
.\]
\end{lm}

\begin{proof}
Si $m \in \Z$ et $n$ est un entier naturel non-nul, on note $\phi_{m,n}$ l'élément de $\A{\mathbb{A}_R}$ (qui s'étend en une fonction holomorphe sur un voisinage de $\bar{\mathbb{A}_R}$) défini par
\[
\phi_{m,n} \(z\) = \sum_{x \in \tilde{A_n}\(z\)} \widetilde{G}_n \(x\) x^m
\]
pour $z \in \mathbb{A}_R$, et la remarque \ref{cont} implique que
\[
\|\phi_{m,n}\| \leqslant D^n \|\widetilde{G}_n\|_{\infty} \( \exp \( \frac{R}{\lambda'^n} \) \)^{|m|}
.\]

Si $m \in \Z$ on note $l_m$ la forme linéaire sur $\A{\mathbb{A}_R}$ définie par
\[
l_m\(\phi\) = \frac{1}{2 i \pi} \s{\psi S^1}{\frac{\phi\(z\)}{z^{m+1}}}{z}
\]
pour $\phi \in \A{\mathbb{A}_R}$ et on remarque que 
\[
\|l_m\| \leqslant e^{-|m|R}
\]
(par la formule de Cauchy, on peut calculer ces intégrales sur des cercles tendant vers les cercles extérieurs). De ces estimations, il découle que pour tout entier naturel non-nul $n$ :
\[
\sum_{m \in \Z} \|l_m\|^{\frac{2}{3}} \|\phi_{m,n}\|^{\frac{2}{3}} < + \infty
\]
et donc la série $\sum_{m \in \Z} l_m \otimes \phi_{m,n}$ converge vers un noyau $\frac{2}{3}$-sommable dans $\A{\mathbb{A}_R}' \oh \A{\mathbb{A}_R}$ . Le lemme précédent et la formule donnant le développement en série de Laurent sur $\mathbb{A}_R$ des éléments de $\A{\mathbb{A}_R}$ assurent que ce noyau représente $\tilde{\L^n}$.
En particulier, $\tilde{\L}$ est $\frac{2}{3}$-sommable. 

De la formule précédente on déduit que pour $n$ entier naturel non-nul on a :
\[
\tr\(\tilde{\L}^n\) = \sum_{m \in \Z} l_m \( \phi_{m,n} \) \\
.\]
Or 
\begin{align*}
\sum_{m \geqslant 0} l_m \( \phi_{m,n} \) & = \sum_{m \geqslant 0} \frac{1}{2i \pi} \s{\psi S^1}{\frac{\phi_{m,n}\(z\)}{z^{m+1}}}{z} \\
        & =  \sum_{m \geqslant 0} \frac{1}{2i \pi} \s{\partial \D{0}{e^R} }{\frac{\phi_{m,n}\(z\)}{z^{m+1}}}{z} \textrm{, par la formule de Cauchy} \\
        & = \sum_{m \geqslant 0} \frac{1}{2i \pi} \s{\partial \D{0}{e^R} }{\frac{\sum_{x \in \tilde{A_n}\(z\)} \widetilde{G}_n\(x\) x^m }{z^{m+1}}}{z}
.\end{align*}
On peut inverser les sommes dans ce dernier terme puisque si $z \in \D{0}{e^R}$ et $x \in \tilde{A_n}\(z\)$ alors $\left|\frac{x^m}{z^{m+1}} \right| \leqslant e^{-R} \( \exp\( m \(\frac{R}{\lambda'^n} - R \)\) \)$ et on a donc convergence normale d'où
\begin{align*}
\sum_{m \geqslant 0} l_m \( \phi_{m,n} \) & = \frac{1}{2i \pi} \s{\partial \D{0}{e^R} }{\sum_{x \in \tilde{A_n}\(z\)} \widetilde{G}_n \(x\) \sum_{m \geqslant 0} \frac{x^m}{z^{m+1}}}{z} \\
     & =\frac{1}{2i \pi} \s{\partial \D{0}{e^R} }{\sum_{x \in \tilde{A_n}\(z\)} \frac{\widetilde{G}_n\(x\)}{z-x}}{z}
.\end{align*}

De la même manière on montre que
\begin{align*}
\sum_{m < 0} l_m \( \phi_{m,n} \) & = \sum_{m < 0} \frac{1}{2 i\pi} \s{\partial \D{0}{e^{-R}} }{\frac{\sum_{x \in \tilde{A_n}\(z\)} \widetilde{G}_n\(x\) x^m }{z^{m+1}}}{z} \\
           & = \frac{1}{2i \pi} \s{\partial \D{0}{e^{-R}} }{\sum_{x \in \tilde{A_n}\(z\)} \widetilde{G}_n \(x\) \sum_{m < 0} \frac{x^m}{z^{m+1}}}{z} \\
           & = - \frac{1}{2i \pi} \s{\partial \D{0}{e^{-R}} }{\sum_{x \in \tilde{A_n}\(z\)} \frac{\widetilde{G}_n\(x\)}{z-x}}{z}
.\end{align*}
Et on a donc 
\[
\tr\(\tilde{\L}^n\) = \frac{1}{2 i \pi} \s{\partial \mathbb{A}_R}{\sum_{x \in \tilde{A_n}\(z\)} \frac{\widetilde{G}_n\(x\)}{z-x}}{z}
.\]

La trace $\tr\(\tilde{\L}^n\)$ est donc la somme des résidus dans $\mathbb{A}_R$ de la forme que l'on intègre. Or celle-ci n'a de pôles qu'en les points de $\fix{\tilde{f}^n}$. Si $x \in \fix{\tilde{f}^n}$ alors l'intégrande de la formule précédente s'écrit comme une somme de $D^n - 1$ fonctions holomorphes et une fonction méromorphe au voisinage de $x$ cette dernière étant $ z \mapsto \frac{\tilde{G}_n\(h\(x\)\)}{z - h\(z\)}$ où $h$ désigne un inverse local de $f^n$ au voisinage de $x$ qui envoie $x$ sur $x$. Le résidu de l'intégrande en $x$ est le même que celui de cette fonction qui n'est autre que :
\[
\frac{\widetilde{G}_n\(x\)}{1-h'\(x\)} = \frac{\(\tilde{f}^n\)' \(x\) \widetilde{G}_n\(x\)}{\(\tilde{f}^n\)'\(x\) -1}
.\]
On en déduit la formule annoncée :
\[
\tr\(\tilde{\L}^n\) = \sum_{x \in \fix{\tilde{f}^n}} \frac{\(\tilde{f}^n\)'\(x\) \widetilde{G}_n\(x\)}{\(\tilde{f}^n\)'\(x\) - 1}
.\]
\end{proof}

\begin{cor}
L'opérateur $\L$ est nucléaire et pour tout entier naturel non-nul $n$ on a 
\[
\tr\(\L^n\) = \sum_{x \in \fix{f^n}} \frac{\(f^n\)'\(x\) G_n\(x\)}{\(f^n\)'\(x\) - 1}
.\]
\end{cor}

\begin{proof}
Il suffit de remarquer que $\psi$ induit une bijection entre $\fix{f^n}$ et $\fix{\tilde{f}^n}$ et que si $x \in \fix{f^n}$ alors $\(f^n\)'\(x\) = \(\tilde{f}^n\)'\(\psi\(x\)\)$.
\end{proof}

\begin{cor}\label{formule}
Si $\rho$ désigne le rayon spectral de $\L$ alors pour tout $z \in \C$ tel que $|z| < \frac{1}{\rho}$ on a :
\[
\det \(I - z \L\) = \exp \( - \sum_{n = 1}^{+ \infty} \frac{z^n}{n} \sum_{x \in \fix{f^n}} \frac{\(f^n\)'\(x\) G_n\(x\)}{\(f^n\)'\(x\) - 1} \)
.\]
En particulier, la série en argument de l'exponentielle converge.
\end{cor}

\begin{rmq}
On retrouve le déterminant de la formule \eqref{annonce} en changeant $G$ en $\frac{G}{f'}$.
\end{rmq}

\subsection{Cas d'une courbe d'applications}\label{12}

On explore maintenant la façon dont les résultats de \S \ref{11} se comportent lorsqu'on effectue une perturbation analytique sur $f$.

On se donne une courbe d'applications analytiques de $S^1$ sur lui-même, $t \in \left]- \epsilon , \epsilon \right[ \mapsto f_t$, où $\epsilon$ est un réel strictement positif. On demande à cette courbe d'être analytique au sens suivant : l'application $\(t,x\) \in \left]- \epsilon , \epsilon \right[ \times S^1 \mapsto f_t\(x\) \in S^1$ est analytique. On demande de plus aux $f_t$ d'être uniformément dilatantes : il existe $\lambda > 1$ tel que pour tout $t \in \left]- \epsilon, \epsilon \right[$ et tout $x \in S^1$ on ait $f_t'\(x\) \geqslant \lambda$.

On s'intéresse à des propriétés de cette courbe au voisinage de $0$ et on s'autorise donc à réduire $\epsilon$ chaque fois que cela sera nécessaire.

Quitte à réduire $\epsilon$, on étend $\(t,x\) \mapsto f_t\(x\)$ en une application holomorphe sur $\D{0}{\epsilon} \times U$, où $U$ est un voisinage de $S^1$ dans $\C / \Z$ tel que pour tout $\(t,x\) \in \D{0}{\epsilon} \times U$ on ait $f_t'\(x\) \neq 0$.

Pour tout $t \in \D{0}{\epsilon}$, on définit $\tilde{f}_t$ et $F_t$ comme dans \S \ref{11}. Notons que lorsque $t$ n'est pas réel, il n'y a pas de raison pour que $f_t$ préserve le cercle. La prise en compte des valeurs complexes de $t$ sera cependant très utile lorsque nous nous intéresserons à la dépendance en $t$ de l'opérateur de transfert.

Enfin, on se donne une application analytique $g : S^1 \to \R$.

Il nous faut maintenant repasser par quelques résultats techniques afin de justifier que le $R<r_0$ de \S \ref{11}, avec $r_0$ comme dans le lemme \ref{antec}, peut être choisi de manière convenable pour tous les $t$ réels suffisamment petits et que la définition de l'opérateur de transfert pour les valeurs non-réelles de $t$ reste valable.

Soit $r$ un réel strictement positif tel que $\mathcal{C}_r \subseteq U$. On note $M$ la borne supérieure de $ |F_t' \(x\)|$ pour $t \in \D{0}{\epsilon}$ et $x \in \C$ tel que $\Re x \in \left[- \frac{r}{2}, 1 + \frac{r}{2} \right]$ et $\Im x \in \left[ - \frac{r}{2}, \frac{r}{2} \right[$, qui est finie quitte à réduire $\epsilon$.

\begin{lm}
Pour $t$ réel, le degré de $f_t$ (vue comme une application du cercle dans lui-même) est indépendant de $t$, on le note $D$.
\end{lm}

\begin{proof}
Le degré de $f_t$ est donnée par la formule $ \s{S^1}{f_t'\(x\)}{x} $ qui est continue en $t$ et donc constante.
\end{proof}

\begin{lm}\label{antec2}
Soit $0 < r_0 < \min\(\frac{r}{4}, \frac{r^2\(\lambda - 1\)}{8M}\)$. Quitte à réduire $\epsilon$, pour tout $t$, tout élément de $\mathcal{C}_{r_0}$ a exactement $D$ antécédents par $f_t$ dans $\mathcal{C}_{r_0}$.
\end{lm}

\begin{proof}
Soient $\lambda' > 1$ tel que $r_0 < \min\(\frac{r}{4}, \frac{r^2\(\lambda - \lambda'\)}{8M}\)$ et $1 < \lambda'' < \lambda'$. Par le lemme \ref{tech}, on sait que pour tout $z \in \partial \mathcal{C}_{r_0}$, $|\Im f_0 \(z\) | \geqslant \lambda' r_0$. Et donc par compacité de $\partial \mathcal{C}_{r_0}$, quitte à réduire $\epsilon$, on a pour tout $t$ et tout $x \in \partial \mathcal{C}_{r_0}$ , $|\Im f_t \(z\) | \geqslant \lambda'' r_0$.

Si $z \in \mathcal{C}_{r_0}$, le nombre d'antécédents de $z$ par $f_t$ dans $\mathcal{C}_{r_0}$ est :
\[
\frac{1}{2i \pi} \s{\partial \mathbb{A}_{r_0}}{\frac{\tilde{f_t}'\(z\)}{\tilde{f_t}\(z\) - \psi \(z\) }}{z}
\]
qui est continue en $t$ car l'intégrande est bornée par $ \frac{\sup_{t} |\tilde{f_t}'\(z\)|}{ 2 \(\sinh\(\lambda'' r_0\) - \sinh\(r_0\)\)}$ qui est finie quitte à réduire $\epsilon$. Ce nombre d'antécédents est donc indépendant de $t$ et vaut $D$ pour $t = 0$ par le lemme \ref{antec}.
\end{proof}

\begin{lm}\label{rev2}
Avec les notations du lemme précédent, pour tout $t \in \D{0}{\epsilon}$, l'application $f_t$ induit un revêtement holomorphe à $D$ feuillets de $\mathcal{C}_{r_0} \cap f_t^{-1} \mathcal{C}_{r_0}$ sur $\mathcal{C}_{r_0}$.
\end{lm}

\begin{proof}
La preuve est le même que celle du lemme \ref{rev}.
\end{proof}

On se donne $r_0$ comme précédemment. On étend $g$ en une application holomorphe sur un voisinage de $S^1$ à valeurs dans $\C$. On se donne dans la suite $R < r_0$ tel que $g$ est définie sur un voisinage de $\bar{\mathcal{C}_R}$.

\begin{df}
Si $t \in \D{0}{\epsilon}$, $u \in \C$,  $\phi \in \A{\mathcal{C}_R}$ et $z \in \bar{\mathcal{C}_R}$ on pose 
\[
\L_{u,t} \phi \(z\) = \sum_{x \in \mathcal{C}_R : f_t\(x\) = z} \frac{e^{- u g\(x\)}}{f_t'\(x\)} \phi\(x\)
.\]
\end{df}

\begin{prop}
Si $t \in \D{0}{\epsilon}$ et $u \in \R$, alors $\L_{u,t}$ définit un opérateur borné de $\A{\mathcal{C}_R}$ dans lui-même.
\end{prop}

\begin{proof}
La preuve est la même que celle du lemme \ref{defin}.
\end{proof}

\begin{lm}
Il existe $R' < R$ tel que pour tout $t$ suffisamment petit et $x \in \bar{\mathcal{C}_R}$ les $d$ antécédents de $x$ par $f_t$ dans $\mathcal{C}_{r_0}$ sont en fait dans $\mathcal{C}_{R'}$.
\end{lm}

\begin{proof}
On définit $\lambda'$ comme dans la preuve du lemme \ref{antec2} et on choisit $R'$ tel que $\frac{R}{\lambda'} < R'< R$. On sait par le lemme \ref{tech} que ce choix de $R'$ convient dans le cas réel. De plus pour tout $t$ réel, si $z \in \partial \mathcal{C}_R'$ alors $f_t\(z\) \notin \bar{\mathcal{C}_R}$. Par compacité de $\partial \mathcal{C}_R'$, c'est encore vrai pour $t$ suffisamment petit et donc si $x \in x \in \bar{\mathcal{C}_R}$ alors le nombre d'antécédents de $x$ par $f_t$ dans $\mathcal{C}_{R'}$ est donné par la formule :
\[
\frac{1}{2i \pi} \s{\partial \mathbb{A}_{R'}}{\frac{\tilde{f}_t'\(z\)}{\tilde{f}_t\(z\)-\psi\(w\)}}{z}
\]
qui est continue en $t$ donc vaut $D$ (car c'est le cas lorsque $t$ est réel). D'où le résultat annoncé.
\end{proof}

Quitte à réduire $\epsilon$, le résultat du lemme précédent a lieu pour tout $t \in \D{0}{\epsilon}$.

On choisit $R' < R'' < R$ et on note $i$ l'application de restriction de $\A{\mathcal{C}_R}$ dans $\A{\mathcal{C}_{R''}}$ et pour $u \in \C$ et $t \in \D{0}{\epsilon}$, $\widehat{\L}_{u,t}$ l'application de $\A{\mathcal{C}_{R''}}$ dans $\A{\mathcal{C}_R}$ définie par la même formule que $\L_{u,t}$, ce qui est légitime par le lemme précédent. On a alors $ \L_{u,t} =\widehat{\L}_{u,t} \circ i$ et $i$ est nucléaire (et même d'ordre $0$), ce sera une conséquence immédiate de la proposition \ref{injnuc} (on peut aussi remarquer que $i$ se factorise par un espace nucléaire).

\begin{lm}\label{Cont}
L'application $\(u,t\) \mapsto \widehat{\L}_{u,t}$ est continue sur $\C \times \D{0}{\epsilon} $.
\end{lm}

\begin{proof}
On commence par remarquer que, quitte à réduire $\epsilon$, il existe une constante positive $C$ tels que pour tous $t,t' \in \D{0}{\epsilon}$ et tout $z \in \bar{\mathcal{C}_R}$ on peut mettre les antécédents dans $\mathcal{C}_R$ de $z$ par $f_t$ en bijection avec ceux par $f_{t'}$ de telle manière que deux points ainsi associés sont à distance plus petite que $C|t-t'|$ (c'est une conséquence du théorème des fonctions implicites et d'un argument de compacité).

Par ailleurs, il est classique de déduire des inégalités de Cauchy que l'application qui à un élément de $\mathcal{C}_{R''}$ associe la restriction de sa dérivée à $\bar{\mathcal{C}_{R'}}$ définit un opérateur borné de $\A{\mathcal{C}_{R''}}$ dans $\A{\mathcal{C}_{R'}}$.

Le résultat se déduit de ces deux faits par un calcul lourd mais sans malice.
\end{proof}

La ruse suivante pour récupérer la dépendance holomorphe de l'opérateur de transfert en $u$ et $t$ est classique.

\begin{lm}\label{cont2}
Soit $z \mapsto \L_z$ une application continue d'un ouvert de $\C^N$ dans l'espace des opérateurs bornés de $\A{\mathcal{C}_{R''}}$ dans $\A{\mathcal{C}_R}$ telle que pour tout $\phi \in \A{\mathcal{C}_{R''}}$ et tout $x \in \mathcal{C}_R$ l'application $z \mapsto \L_z \phi \(x\)$ est holomorphe. Alors $z \mapsto \L_z$ est holomorphe.
\end{lm}

\begin{proof}
L'hypothèse de continuité permet de donner un sens à la formule de Cauchy. L'autre hypothèse assure qu'elle est vérifiée.
\end{proof}

On en déduit alors immédiatement avec le lemme \ref{Cont}:

\begin{lm}
L'application $\(u,t\) \mapsto \widehat{\L}_{u,t}$ est holomorphe sur $\C \times \D{0}{\epsilon}$.
\end{lm}

\begin{cor}\label{hol}
Il existe une application holomorphe $\(u,t\)  \mapsto L_{u,t} \in \A{\mathcal{C}_R}' \oh \A{\mathcal{C}_R} $ sur $\C \times \D{0}{\epsilon}$ telle que $L_{u,t}$ est un noyau $\frac{2}{3}$-sommable de $\L_{u,t}$ pour tout $\(u,t\) \in \C \times \D{0}{\epsilon}$.
\end{cor}

\begin{proof}
On choisit un noyau $\frac{2}{3}$-sommable $v$ de $i$ et on pose $L_{u,t} = \widehat{\L}_{u,t} \circ v $.
\end{proof}

\begin{prop}\label{ppal}
Si $z \in \C$, $u \in \R$ et $t \in \left]- \epsilon, \epsilon \right[$, on note :
\[
d\(z,u,t\) = \det\(I - z \L_{u,t}\)
.\]
L'application $d$ ainsi définie est analytique : c'est la restriction d'une application holomorphe par le corollaire \ref{hol}.
\end{prop}

\begin{rmq}
On ne considérera dans la prochaine partie que les valeurs réelles de $u$ et $t$. Le passage par des valeurs complexes n'était qu'un artifice pour récupérer l'analyticité de $d$ en utilisant le lemme \ref{cont2}, qui nous servira également à obtenir une estimation de la vitesse de convergence des séries apparaissant dans la formule que nous donnerons pour la réponse linéaire.
\end{rmq}

Nous aurons besoin d'un dernier résultat sur le déterminant dynamique avant de pouvoir donner la formule promise pour la réponse linéaire.

Si $t \in \left]- \epsilon, \epsilon \right[$ et $u \in \R$ et $\phi$ est une application höldérienne de $S^1$ dans $\R$, on note $P_t\(\phi\)$ la pression topologique de $\phi$ pour la dynamique de $f_t$ (voir \S 3.4 de \cite{Ruelle2} pour une définition).

Si $t \in \left]- \epsilon, \epsilon \right[$ et $u \in \R$, on note $\check{\L}_{u,t}$ l'opérateur de transfert associé à la dynamique $f_t$ et au poids $\frac{e^{ug}}{f_t'}$ agissant sur l'espace des fonctions $\mathcal{C}^1$ de $S^1$ dans $\C$. On note $j$ l'opérateur de restriction de $\A{\mathcal{C}_R}$ dans cet espace. On a alors $j \circ \L_{u,t} = \check{\L}_{u,t} \circ j$ et par conséquent le spectre de $\L_{u,t}$ est inclus dans celui de $\check{\L}_{u,t}$ (multiplicités prises en compte), puisque $j$ est injective.

\begin{prop}\label{trouspec}
Si $u \in \R$ et $t \in \left]- \epsilon, \epsilon \right[$ alors $\exp\(P_t\(-\log f_t ' -ug\)\)$ est valeur propre de $\L_{u,t}$. De plus, elle est simple et $\L_{u,t}$ n'a pas d'autres valeurs propres hors d'un disque de centre $0$ et de rayon strictement plus petit que $\exp\(P_t\(-\log f_t ' -ug\)\)$ (on dit que $\L_{u,t}$ a un trou spectral associé à la valeur propre $\exp\(P_t\(-\log f_t ' -ug\)\)$).
\end{prop}

\begin{proof}
Le résultat annoncé est vérifié par $\check{\L}_{u,t}$ (voir par exemple le théorème 2.6 page 93 de \cite{Bal1}) et donc, par le point ci-dessus, il suffit de montrer que $\exp\(P_t\(-\log f_t ' -ug\)\)$ est effectivement valeur propre de $\L_{u,t}$. 

Pour la même raison, la suite $\(e^{ -n P_t\(-\log f_t ' - ug\)} \check{\L}_{u,t}^n \( j \(\mathbbm{1}\)\)\)_{n \geqslant1}$ converge dans l'espace des fonctions $\mathcal{C}^1$ (en particulier simplement) vers un vecteur propre $h$ de $\check{\L}_{u,t}$ associé à cette valeur propre. 

Notons alors $\rho$ le rayon spectral de $\L_{u,t}$ et supposons que $\exp\(P_t\(-\log f_t ' -ug\)\)$ n'est pas valeur propre de $\L_{u,t}$ . Le trou spectral de $\check{\L}_{u,t}$ implique alors que $\rho < \exp\(P_t\(-\log f_t ' -ug\)\)$ et donc la suite $\(e^{ -n P_t\(-\log f_t ' -ug\)} \L_{u,t}^n \( \mathbbm{1}\)\)_{n \geqslant1}$ converge vers $0$ dans $\A{\mathcal{C}_R}$, ce qui est absurde puisqu'elle converge simplement vers $h \neq 0$ sur le cercle.
\end{proof}

Si $n$ est un entier naturel, $t \in \left]- \epsilon, \epsilon \right[$ et $x \in S^1$, on pose 
\[
g_{t,n}\(x\) = \sum_{k=0}^{n-1} g\(f_t^k\(x\)\)
.\]

\begin{prop}\label{final}
Si $u \in \R$ et $t \in \left]- \epsilon, \epsilon \right[$ alors $\exp\(-P_t\(-\log f_t ' -ug\)\)$ est le zéro de plus petit module de l'application holomorphe $ z \mapsto d\(z,u,t\)$. De plus, il est simple et si $|z| < \exp\(-P_t\(-\log f_t ' -ug\)\)$ alors la série 
\[
\sum_{n \geqslant 1 } \frac{z^n}{n} \sum_{x \in \fix{f_t^n}} \frac{e^{-u g_{t,n}\(x\)}}{\(f_t^n\)'\(x\) - 1}
\]
converge et on a
\[
d\(z,u,t\) = \exp \( - \sum_{n = 1}^{+ \infty} \frac{z^n}{n} \sum_{x \in \fix{f_t^n}} \frac{e^{-u g_{t,n}\(x\)}}{\(f_t^n\)'\(x\) - 1} \)
.\]
\end{prop}

\begin{proof}
C'est une conséquence immédiate de la proposition \ref{trouspec} et du corollaire \ref{formule}.
\end{proof}

\subsection{Formule pour la réponse linéaire}\label{13}

Pour tout $t \in \left] - \epsilon, \epsilon \right[$, on note $\mu_t$ l'unique mesure invariante par $f_t$ absolument continue devant Lebesgue. On note également $\mu = \mu_0$. Nous commençons par rappeler un résultat classique (voir par exemple \S 7.28 et \S 7.12 dans \cite{Ruelle2}).

\begin{prop}\label{class}
Pour tout $t \in \left] - \epsilon, \epsilon \right[$ on a :
\[
\left. \frac{\partial}{\partial u} \( P_t\(- \log f_t + u g\) \) \right|_{u=0} = \s{S^1}{g}{\mu_t}
.\]
\end{prop}

\begin{df}
Pour tous $u \in \R$ et $t \in \left]- \epsilon, \epsilon \right[$ on pose 
\[
z\(u,t\) = \exp\(-P_t\(-\log f_t ' -ug\)\)
.\]
L'application ainsi définie est analytique par le théorème des fonctions implicites et la proposition \ref{final}.
\end{df}

\begin{rmq}
Si $t \in \left]- \epsilon, \epsilon \right[$, alors il est classique que $z\(0,t\) = 1$ (voir par exemple \cite{Mane}, page 230).
\end{rmq}

Nous rappelons maintenant l'argument de \cite{Poll} pour établir une formule pour la réponse linéaire à partir de la fonction $d$.

\begin{prop}\label{expr}
Pour tout $t \in \left]- \epsilon, \epsilon \right[$, on a :
\[
\s{S^1}{g}{\mu_t} = - \frac{\drond{d}{u}\(1,0,t\)}{\drond{d}{z}\(1,0,t\)}
.\]
\end{prop}

\begin{proof}
Dérivant l'équation $d\(z\(u,t\),u,t\) = 0$ par rapport à $u$, il vient
\[
\drond{z}{u}\(0,t\) \drond{d}{z}\(1,0,t\) + \drond{d}{u}\(1,0,t\) = 0
\]
or, par la proposition \ref{class}, on a
\[
\drond{z}{u}\(0,t\) = z\(0,t\) \(\s{S^1}{g}{\mu_t} \) = \s{S^1}{g}{\mu_t}
\]
et
\[
\drond{d}{z}\(1,0,t\) \neq 0
\]
puisque $1$ est racine simple de $z \mapsto d\(z,0,t\)$.
\end{proof}

\begin{cor}
L'application $t \mapsto \s{S^1}{g}{\mu_t}$ est analytique.
\end{cor}

On déduit immédiatement de la proposition \ref{expr}, par dérivation, une formule pour la réponse linéaire.

\begin{cor}\label{derivee}
La dérivée de l'application $t \mapsto \s{S^1}{g}{\mu_t}$ est
\[
t \mapsto - \frac{\ddeux{d}{u}{t}\(1,0,t\)}{\drond{d}{z}\(1,0,t\)} + \frac{\ddeux{d}{z}{t}\(1,0,t\) \drond{d}{u}\(1,0,t\)}{\(\drond{d}{z}\(1,0,t\) \)^2}  
\]
ou encore
\[
t \mapsto - \frac{\ddeux{d}{u}{t}\(1,0,t\)}{\drond{d}{z}\(1,0,t\)} - \s{S^1}{g}{\mu_t} \frac{\ddeux{d}{z}{t}\(1,0,t\)}{\drond{d}{z}\(1,0,t\)}
.\]
\end{cor}

Nous allons maintenant expliciter plus avant les formules du corollaire \ref{derivee}, jusqu'à un point qui permettrait un calcul numérique(voir \S 4 de \cite{Poll}).

On note pour $t \in \left]- \epsilon, \epsilon \right[$, $u \in \R$ et $n$ entier naturel non-nul
\begin{equation}\label{un}
b_n \(u,t\) = \sum_{x \in \fix{f_t^n}} \frac{e^{-u g_{t,n}\(x\)}}{\(f_t^n\)'\(x\) - 1} 
\end{equation}
et on écrit
\begin{equation}\label{developpement}
d\(z,u,t\) = \sum_{n=0}^{+ \infty} a_n\(u,t\) z^n.
\end{equation}
On a alors
\begin{equation}
a_0\ = 1, \drond{a_0}{u} = 0, \ddeux{a_0}{u}{t}  = 0
\end{equation}
et pour tout entier naturel non-nul $n$ :
\begin{equation}
a_n  = - \frac{1}{n} \sum_{j=0}^{n-1} a_j b_{n-j},
\end{equation}
\begin{equation}
\drond{a_n}{u}  = - \frac{1}{n} \sum_{j=0}^{n-1} \(\drond{a_j}{u} b_{n-j}+ a_j \drond{b_{n-j}}{u} \),
\end{equation}
\begin{equation}
\drond{a_n}{t}  = - \frac{1}{n} \sum_{j=0}^{n-1} \(\drond{a_j}{t} b_{n-j}+ a_j \drond{b_{n-j}}{t} \),
\end{equation}
et
\begin{equation}
\ddeux{a_n}{u}{t}  = - \frac{1}{n} \sum_{j=0}^{n-1} \( \ddeux{a_j}{u}{t} b_{n-j}  + \drond{a_j}{u}\drond{b_{n-j}}{t} + \drond{a_j}{t} \drond{b_{n-j}}{u} + a_j \ddeux{b_{n-j}}{u}{t} \)
.\end{equation}

\begin{cor}\label{formula}
La formule du corollaire \ref{derivee} devient avec ces notations :
\begin{equation}\label{sommes}
\left. \drond{}{t} \( \s{S^1}{g}{\mu_t} \) \right|_{t=0} = -\frac{\sum_{n=1}^{+ \infty} \ddeux{a_n}{u}{t} \(0,0\) }{\sum_{n=1}^{+\infty} n a_n\(0,0\)} + \frac{\(\sum_{n=1}^{+ \infty} n \drond{a_n}{t} \(0,0\)  \)\( \sum_{n=1}^{+ \infty} \drond{a_n}{u}\(0,0\) \)}{\(\sum_{n=1}^{+\infty} n a_n\(0,0\)\)^2}
.\end{equation}
\end{cor}

Ce qui permet un calcul explicite dès que l'on sait calculer $b_n, \drond{b_n}{u}, \drond{b_n}{t}$ et $\ddeux{b_n}{u}{t}$ en $\(0,0\)$.

Pour tout entier naturel $n$ non-nul et tout $x \in S^1$ on note $X_n\(x\)$ la dérivée en $0$ de l'application $t \mapsto f_t^n \(x\)$, de manière moins rigoureuse :
\[
X_n \(x\) = \left. \drond{}{t} \( f_t^n\(x\) \) \right|_{t=0}
.\]

En notant $X = X_1$, on a alors pour tout entier naturel non-nul $n$ la formule 
\begin{equation}
X_n = \sum_{k=0}^{n-1} X \circ f_0^k \(f_0^{n-1-k}\)' \circ f_0^k
.\end{equation}

Pour calculer les dérivées de $b_n$, on a les formules :
\begin{equation}
\drond{b_n}{u} \(0,t\) = - \sum_{x \in \fix{f_t^n}} \frac{g_{t,n}\(x\)}{\(f_t^n\)'\(x\) - 1}
\end{equation}
\begin{equation}
\drond{b_n}{t}\(0,0\) = - \sum_{x \in \fix{f_0^n}} \frac{X_n'\(x\) + \(f_0^n\)''\(x\) \frac{X_n\(x\)}{1-\(f_0^n\)'\(x\)}}{\(\(f_0^n\)'\(x\)-1\)^2}
\end{equation}
et
\begin{equation}
\ddeux{b_n}{u}{t}\(0,0\) = \sum_{x \in \fix{f_0^n}}\( \frac{X_n' + \(f_0^n\)'' \frac{X_n}{ 1-\(f_0^n\)'}}{\(\(f_0^n\)' - 1 \)^2} g_{0,n} + \frac{1}{\(f_0^n\)' - 1} \sum_{k=0}^{n-1} \frac{X_n \circ f_0^k g' \circ f_0^k}{ \(f_0^n\)' \circ f_0^k-1}\)\(x\)
.\end{equation}

Enfin, on termine cette sous-partie par un résultat qui dans certains cas simples donne une autre méthode de calcul pour la réponse linéaire. N'ayant pas réussi à comprendre la preuve du théorème 4.1 de \cite{Poll} nous donnons ici un énoncé alternatif.

\begin{lm}
Supposons qu'il existe un réel $A$ tel que
\[
\s{S^1}{g}{\mu} = 0,
\]
et les séries
\[
\sum_{n \geqslant 1} \(\frac{1}{n} \ddeux{b_n}{u}{t} \(0,0\) - A\)
\textrm{ et }
\sum_{n \geqslant 1} \frac{1}{n} \drond{b_n}{u} \(0,0\)
\]
convergent (où les $b_n$ sont ceux définis par \eqref{un}). Alors
\[
\left. \drond{}{t} \(\s{S^1}{g}{\mu_t} \) \right|_{t=0} = - A = - \lim_{n \to + \infty} \frac{1}{n} \ddeux{b_n}{u}{t} \(0,0\)
.\]

\end{lm}

\begin{proof}
La deuxième formule donnée dans le corollaire \ref{derivee} assure que 
\[
\left. \drond{}{t} \(\s{S^1}{g}{\mu_t} \) \right|_{t=0} = - \frac{\ddeux{d}{u}{t}\(1,0,0\)}{\drond{d}{z}\(1,0,0\)}
.\]
Or pour $|z| <1$ on a par le calcul
\begin{align*}
\ddeux{d}{u}{t}\(z,0,0\)  & = - \( \sum_{n \geqslant 1} \ddeux{b_n}{u}{t}\(0,0\) \frac{z^n}{n} \) d\(z,0,0\) - \( \sum_{n \geqslant 1} \drond{b_n}{u}\(0,0\)\frac{z^n}{n}\) \drond{d}{t}\(z,0,0\) \\
 & = A \frac{d\(z,0,0\)}{z-1} - \( \sum_{n \geqslant 1} \( \frac{1}{n} \ddeux{b_n}{u}{t}\(0,0\) - A \)z^n \) d\(z,0,0\) \\ & - \( \sum_{n \geqslant 1} \drond{b_n}{u}\(0,0\)\frac{z^n}{n}\) \drond{d}{t}\(z,0,0\). \\
\end{align*}
Or $d\(1,0,0\) = \drond{d}{t}\(1,0,0\) = 0$ et les sommes qui apparaissent en facteur ont une limite en $z=1$ (par valeurs réelles) par le théorème d'Abel et donc par passage à la limite :
\[
\ddeux{d}{u}{t}\(1,0,0\) = A \drond{d}{z}\(1,0,0\)
.\]
D'où le résultat annoncé.

\end{proof}

\subsection{Vitesse de convergence}\label{14}

Nous allons explorer la vitesse à laquelle les $a_n$ (définis par \eqref{developpement}) et leurs dérivées tendent vers $0$. On se base pour cela sur une remarque de \cite{Groth}, page 64 de la deuxième partie, rectifiée par Fried dans \cite{fried} (travaillant en dimension $1$, cette rectification n'a pas de conséquence ici).

\begin{prop}\label{injnuc}
Il existe une suite $\(\phi_n\)_{n \in \N}$ d'éléments de $\A{\mathcal{C}_{R''}}$, une suite $\(l_n\)_{n \in \N}$ d'éléments de $\A{\mathcal{C}_R}'$, et des constantes $C > 0$ et $0 < \eta <1$ telles que 
\[
\forall n \in \N : \|\phi_n\| \|l_n\| \leqslant C \eta^n
\textrm{ et }
\sum_{n \in \N} l_n \otimes \phi_n
\]
est un noyau de $i$, l'injection de $\A{\mathcal{C}_R}$ dans $\A{\mathcal{C}_{R''}}$ introduite dans la deuxième sous-partie.
\end{prop}

\begin{proof}
Il suffit de montrer un résultat analogue pour l'opérateur $\iota$ de restriction de $\A{\mathbb{A}_R}$ dans $\A{\mathbb{A}_{R''}}$. Pour tout entier $m$ on note $\phi_m$ l'élément de $\A{\mathbb{A}_{R''}}$ qui à $z$ associe $z^m$ et $l_m$ l'élément de $\A{\mathbb{A}_{R''}}$ qui à $\phi$ associe 
\[
l_m \(\phi\) = \frac{1}{2i \pi} \s{\partial \psi S^1}{\frac{\phi \(z\)}{z^{m+1}}}{z}
.\]
On a alors pour tout entier $m$ :
\[
\|\phi_m\| \leqslant e^{2 \pi |m| R''}
\]
et
\[
\|l_m\| \leqslant e^{- 2 \pi |m| R}
\]
(en calculant l'intégrale sur les cercles extérieurs comme dans la preuve de \ref{calc}).

La formule du développement en série de Laurent d'une fonction définie sur un anneau assure que
\[
\sum_{m \in \Z} l_m \otimes \phi_m
\]
est bien un noyau de $\iota$.
\end{proof}

\begin{prop}\label{conv1}
Soit $\mathcal{B}$ un Banach et $u \in \mathcal{B}' \oh \mathcal{B}$ tel qu'il existe une suite $\(e_n\)_{n \in \N}$ de $\mathcal{B}$ de norme $1$, une suite $\(f_n\)_{n \in \N}$ d'éléments de $\mathcal{B}'$ de norme $1$, une suite $\(\lambda_n\)_{n \in \N}$ de nombres complexes et des constantes $C >0$ et $0 < \eta <1$ telles que 
\[
\forall n \in \N : |\lambda_n| \leqslant C \eta^n
\]
et
\[
v = \sum_{n \in \N} \lambda_n f_n \otimes e_n
.\] 

Si on écrit 
\[
\det\(I + zv\) = \sum_{n \in \N} a_n z^n
\]
pour $z \in \C$, on a alors 
\[
\forall n \in \N : |a_n| \leqslant C^n \frac{\eta^{\frac{n\(n-1\)}{2}}}{\prod_{k=1}^n \(1 - \eta^k\)} n^{\frac{n}{2}}
.\]
\end{prop}

\begin{proof}
Pour tout entier naturel $n$ on a (voir par exemple la page 17 de la deuxième partie de \cite{Groth})
\[
a_n = \sum_{0 \leqslant i_1 < \dots < i_n} \lambda_{i_1} \dots \lambda_{i_n} \det\(\(f_{i_p} \(e_{i_q}\)\)_{1 \leqslant p,q \leqslant n} \)
\]
et donc par le théorème d'Hadamard
\[
|a_n| \leqslant \sum_{0 \leqslant i_1 < \dots < i_n } C^n \eta^{i_1 + \dots + i_n} n^{\frac{n}{2}} = C^n \frac{\eta^{\frac{n\(n-1\)}{2}}}{\prod_{k=1}^n \(1 - \eta^k\)} n^{\frac{n}{2}}
.\]
\end{proof}

La proposition \ref{injnuc} assure que les opérateurs $\L_{u,t} = \widehat{\L}_{u,t} \circ i$ vérifient les hypothèses de la proposition \ref{conv1} pour une même valeur de $C$ lorsque $u$ et $t$ sont au voisinage de $0$, on obtient donc le résultat suivant.
\begin{prop}\label{vitconv}
Les $a_n$ définis par \eqref{developpement} vérifient
\[
a_n\(u,t\) \underset{n \to + \infty}{=} O \(\theta^{n^2}\)
\]
uniformément en $u$ et en $t$ dans un voisinage de $\(0,0\)$ dans $\C^2$, pour un certain $0< \theta <1$. Les dérivées partielles des $a_n$ vérifient la même estimation. 
\end{prop}
Le résultat sur les dérivées est obtenu par la formule de Cauchy. On a donc la même estimation sur les restes des séries dont les sommes apparaissent dans la formule \eqref{sommes} du corollaire \ref{formula} puisque 
\[
\sum_{k=n}^{+ \infty} \theta^{k^2} \underset{ n \to \infty}{\sim} \theta^{n^2}
.\]

\section{Difféomorphismes d'Anosov du tore}

Comme dans la partie précédente, il s'agit ici de définir un déterminant dynamique et de démontrer qu'il varie analytiquement en chacun de ses arguments, c'est le théorème \ref{gros} qui détaille la première partie du lemme 2.13 de \cite{Poll}.

La partie \ref{21} est dévouée à la définition du déterminant dynamique dans un cadre plus symbolique que celui qui nous intéresse, rappelant la construction de \cite{R1}. Nous montrons dans \S \ref{22} comment se comporte dans ce cadre le déterminant sous des perturbations analytiques.

Suivant \cite{R2}, nous appliquons dans \S \ref{23} ces résultats au cas des difféomorphismes d'Anosov du tore et nous démontrons le théorème \ref{gros}. Un argument essentiel sera l'usage de partitions de Markov et de la formule donnée par Manning dans \cite{Manning} pour compter les points périodiques d'un difféomorphisme Axiome A.

La construction de \S \ref{23} ne mettant pas en évidence l'entièreté du déterminant dynamique, nous en donnons une idée de preuve, inspirée de \cite{R2}, dans \S \ref{24}.

Enfin, \S \ref{25} est dédiée à une estimation de la vitesse à laquelle les coefficients du développement en série entière du déterminant dynamique et leurs dérivées tendent vers $0$ (proposition \ref{vitconv2}), donnant un résultat plus faible que la deuxième partie du lemme 2.13 de \cite{Poll}.

\subsection{Systèmes d'applications analytiques réelles hyperboliques}\label{21}

Cette sous-partie rappelle la construction faite par Rugh dans \cite{R1}. Certains points ont dû être très légèrement adaptés en prévision de la troisième sous-partie. Les preuves restent très proches de celles de \cite{R1} et sont données en annexe. 

\begin{df}
Soient $D_1,D_2,D_1'$ et $D_2'$ des disques dans $\C$. Une application $f$ holomorphe sur un voisinage de $D_1 \times D_2$ dans $\C^2$ à valeurs dans $\C^2$ est dite analytique réelle hyperbolique de $D_1 \times D_2$ sur $D_1' \times D_2'$ si, en notant $f = (f_1,f_2)$ on a :
\begin{itemize}
\item $f_1 \( \bar{D_1} \times \bar{D_2} \) \subset \bul{D_1'}$;
\item pour tout $w_1$ dans un voisinage ouvert $U_1$ de $\bar{D_1}$ et tout $z_2 $ dans un voisinage ouvert $U_2$ de $\bar{D_2'}$ il existe un unique $w_2 \in \bul{D_2} $ tel que $f_2\(w_1,w_2\) = z_2$. 
\end{itemize}
\end{df}

On utilisera le lemme suivant à maintes reprises dans la suite. L'usage de résultats similaires est classique, on en trouvera des exemples dans \cite{R1} mais aussi \cite{Ruelle1}. Le preuve est en annexe \ref{a1}.

\begin{lm}\label{pf}
Soit $U,D$ des ouverts de $\C^m$, $m$ entier naturel non nul, et $K$ une partie compacte de $\C^m$ tels que $K \subset U$ et $ \bar{U} \subset D$, $U$ étant connexe et borné. Soit $N$ un entier naturel et $V$ un ouvert de $\C^N$. Pour tout $x \in V$, on se donne une application holomorphe $\psi_x$ sur $D$ qui envoie $U$ dans $K$ de telle manière que l'application $\(x,z\) \mapsto \psi_x\(z\)$ soit holomorphe. Alors pour tout $x \in V$, l'application $\psi_x$ admet un unique point fixe $z_x$ dans $\bar{U}$ qui est en fait dans $U$ et dont la dépendance en $x$ est holomorphe. De plus, les valeurs propres de la différentielle de $\psi_x$ en $z_x$ sont de module strictement plus petit que $1$.
\end{lm}

\begin{df}
Un système d'applications analytiques réelle hyperboliques est la donnée :
\begin{itemize}
\item d'un ensemble non-vide $I$;
\item d'une matrice $a : I \times I \to \set{0,1}$;
\item pour tout $i \in I$ de deux disques dans $\C$, $D_1^i$ et $D_2^i$;
\item pour tout couple d'éléments $\(i,j\) \in I \times I$ tel que $a\(i,j\) = 1$ d'une application $f^{ij} = \(f_1^{ij},f_2^{ij}\)$ analytique réelle hyperbolique de $D_1^i\times D_2^i$ sur $D_1^j \times D_2^j$.
\end{itemize}
\end{df}

Dans toute la suite de cette partie, on se donne un système d'applications analytiques réelle hyperboliques avec les notations ci-dessus. Pour tout $\(i,j\) \in I \times I$ tel que $a\(i,j\) = 1$, on note 
\[
\phi_s^{ij} : \bar{D_1^i} \times \bar{D_2^j} \to \bul{D_2^i}
\]
qui à $w_1 \in \bar{D_1^i} $ et $z_2 \in \bar{D_2^j}$ associe l'unique solution $w_2 \in \bul{D_2^i}$ de $ f_2^{ij}\(w_1,w_2\) = z_2$. L'application $\phi_s^{ij}$ s'étend en une application holomorphe au voisinage de $\bar{D_1^i} \times \bar{D_2^j}$, toujours noté $\phi_s^{ij}$, par un raisonnement d'analyse complexe élémentaire. Pour $w_1 \in \bar{D_1^i} $ et $z_2 \in \bar{D_2^j}$ on note $\phi_u^{ij} \(w_1,z_2\) = f_1^{ij}\(w_1, \phi_s^{ij} \(w_1,z_2\)\) \in \bul{D_1^j}$ de telle manière que 
\[
f^{ij} \(w_1, \phi_s^{ij} \(w_1,z_2\)\) = \( \phi_u^{ij} \(w_1,z_2\) , z_2 \)
\]
et $\phi_u^{ij}$ s'étend en une application holomorphe au voisinage de $\bar{D_1^i} \times \bar{D_2^j}$.

Pour toute suite $i_0, \dots, i_n$ d'éléments de $I$ telle que $a\(i_j,i_{j+1}\)=1$ pour $j=1, \dots,n$, on note :
\[
f^{i_0 \dots i_n} = f^{i_{n-1} i_n} \circ f^{i_{n-2} i_{n-1}} \circ \dots \circ f^{i_0 i_1}
\]
(avec la convention que $f^{i_0} = I$) et on va définir par récurrence une application holomorphe $ \phi_s^{i_0 \dots i_n}$ sur un voisinage de $\bar{D_1^{i_0}} \times \bar{D_2^{i_n}}$ qui envoie $\bar{D_1^{i_0}} \times \bar{D_2^{i_n}}$ dans $\bul{D_2^{i_0}}$ et une application holomorphe $ \phi_u^{i_0 \dots i_n}$ sur un voisinage de $\bar{D_1^{i_0}} \times \bar{D_2^{i_n}}$ qui envoie $\bar{D_1^{i_0}} \times \bar{D_2^{i_n}}$ dans $\bul{D_1^{i_n}}$ et qui vérifieront en particulier
\begin{equation}\label{param}
f^{i_0 \dots i_n} \circ \(pr_1, \phi_s^{i_0 \dots i_n}\) = \(\phi_u^{i_0 \dots i_n}, pr_2 \)
,\end{equation}
où $pr_1$ et $pr_2$ désignent les projections respectivement sur les première et deuxième coordonnées.

Ces applications fournissent ce que Rugh appelle des \og pinning coordinates \fg{} dans \cite{R1} : l'application
\[
\(w_1,z_2\) \in \bul{D_1^{i_0}} \times \bul{D_2^{i_n}} \mapsto \(w_1,\phi_s^{i_0 \dots i_n}\(w_1,z_2\)\)
\]
paramètre les trajectoires pour $f^{i_0 i_1}, \dots , f^{i_{n-1} i_n}$ qui restent dans les rectangles qu'on s'est donnés.

Pour le cas $n=1$, on vient de définir ces applications. Pour le pas de récurrence si $i_0,\dots,i_{n+1}$ est tel que $a\(i_j,i_{j+1}\)=1$ pour $j=1, \dots,n+1$ et que $ \phi_s^{i_0 \dots i_n}$ et $ \phi_u^{i_0 \dots i_n}$ sont définies, et si $\(w_1,z_2\)$ est dans un voisinage de $\bar{D_1^{i_0}} \times \bar{D_2^{i_{n+1}}}$ on note $\xi^{i_0 \dots i_{n+1}} \(w_1,z_2\) = \(\xi_1^{i_0 \dots i_{n+1}} \(w_1,z_2\),\xi_2^{i_0 \dots i_{n+1}} \(w_1,z_2\)\) $ l'unique point fixe de l'application 
\[
y_1,y_2 \mapsto \( \phi_u^{i_0 \dots i_n} \(w_1, y_2 \) , \phi_s^{i_n i_{n+1}} \(y_1,z_2\) \)
\]
qui, par le lemme \ref{pf}, est bien défini et est analytique en $\(w_1,z_2\)$ puisque cette application est holomorphe et envoie $\bar{D_1^{i_n}} \times \bar{D_2^{i_n}}$ dans $\bul{D_1^{i_n}} \times \bul{D_2^{i_n}}$ et on pose
\begin{equation}\label{def1}
\phi_u^{i_0 \dots i_{n+1}} \(w_1,z_2\) = \phi_u^{i_n i_{n+1}} \(\xi_1^{i_0 \dots i_{n+1}} \(w_1,z_2\),z_2\)
\end{equation}
et
\begin{equation}\label{def2}
\phi_s^{i_0 \dots i_{n+1}} \(w_1,z_2\) = \phi_s^{i_0 \dots i_n} \( w_1,\xi_2^{i_0 \dots i_{n+1}} \(w_1,z_2\)\)
.\end{equation}

Notons qu'on a alors
\begin{align*}
f^{i_n i_{n+1}} \( \phi_u^{i_0 \dots i_n} \(w_1,x_2\),x_2\) & = f^{i_n i_{n+1}} \( \phi_u^{i_0 \dots i_n} \(w_1,x_2\),\phi_s^{i_n i_{n+1}} \(x_1,z_2\)\) \\
      & = f^{i_n i_{n+1}} \( x_1,\phi_s^{i_n i_{n+1}} \(x_1,z_2\)\) \\
      & = \( \phi_u^{i_n i_{n+1}} \(x_1,z_2\), z_2\) \\
      & = \( \phi_u^{i_0 \dots i_{n+1}} \(w_1,z_2\),z_2\) 
\end{align*}
et donc par récurrence il vient
\[
f^{i_0 \dots i_{n+1}} \( w_1, \phi_s^{i_0 \dots i_{n+1}} \(w_1,z_2\)\) = \( \phi_u^{i_0 \dots i_{n+1}} \(w_1,z_2\),z_2\)
.\]

Si $i_1, \dots, i_n$ sont des éléments de $I$ tels que $a\(i_0,i_1\) = \dots = a\(i_{n-1},i_n\) = a\(i_n,i_0\)= 1$ alors par le lemme \ref{pf}, l'application $ w_1,w_2 \mapsto \(\phi_u^{i_0 \dots i_n i_0} \(w_1,w_2\) , \phi_s^{i_0 \dots i_n i_0} \(w_1,w_2\)\)$ a un unique point fixe que l'on note $x^{i_0 \dots i_n i_0} = \(x_1^{i_0 \dots i_n i_0},x_2^{i_0 \dots i_n i_0}\)$.  Alors $x^{i_0 \dots i_n i_0}$ est l'unique point fixe de $f^{i_0 \dots i_n i_0}$ dont l'orbite reste dans les rectangles qu'on s'est donnés (par la formule \eqref{param}). Notons par ailleurs que ces points fixes sont hyperboliques (voir \cite{R1}).

Sans restriction de généralité, on suppose que pour tout $i \in I$ et tout $j \in \set{1,2}$ le disque $D_j^i$ est centré en $0$ et on note $r_j^i$ son rayon. On note $\hat{\C}= \C \cup \set{\infty}$ la sphère de Riemann. Pour tout $i \in I$, on note $U_i$ l'ouvert $\(\hat{C} \setminus \bar{D_1^i}\) \times \bul{D_i^2}$ , $V_i$ l'ouvert $\bul{D_1^i} \times \( \hat{\C} \setminus \bar{D_2^i} \)  $ et $\B_i = \A{U_i}$ l'espace des fonctions holomorphes sur $U_i$ qui se prolonge par continuité à la frontière, muni de la norme du sup. On note $\B = \prod_{i \in I} \B_i$. Pour tout $i \in I$ on note $\pi_i$ la projection canonique de $\B$ sur $\B_i$ et $b_i$ l'injection de $\B_i$ dans $\B$.

On se donne dans la suite une application holomorphe $h$ définie sur un voisinage de l'union disjointe des $\bar{D_1^i} \times \bar{D_2^i}$ pour $i \in I$. On se donne également pour tous $i,j \in I$ tels que $a\(i,j\) = 1$ un signe $s_{ij} \in \set{+1,-1}$. Pour tous $i_0, \dots, i_n \in I$ tels que $a\(i_0,i_1\) = \dots = a\(i_{n-1},i_n\) = 1$ on pose alors 
\[
s_{i_0 \dots i_n} = \prod_{j=0}^{n-1} s_{i_j i_{j+1}}
\]
et
\[
h^{i_0 \dots i_n} \(z,w\) = \prod_{k=1}^n h \(f^{i_0 \dots i_k} \(w_1, \phi_s^{i_0 \dots i_n} \(w_1,z_2\)\)\)
\]
pour $w = \(w_1,w_2\)$ dans un voisinage de $\bar{V_{i_0}}$ et $z = \(z_1,z_2\)$ dans un voisinage de $\bar{U_{i_n}}$.

On pose alors
\[
G^{i_0 \dots i_n} \(z,w\) = s_{i_0 \dots i_n} \frac{h^{i_0 \dots i_n }\(z,w\)}{z_1 - \phi_u^{i_0 \dots i_n}\(w_1,z_2\)} \frac{\partial_2 \phi_s^{i_0 \dots i_n} \(w_1,z_2\)}{w_2 - \phi_s^{i_0 \dots i_n }\(w_1,z_2\)}
\]
pour $z = \(z_1,z_2\)$ dans un voisinage de $\bar{U_{i_n}}$ et $w = \(w_1,w_2\)$ dans un voisinage de $\bar{V_{i_0}}$.

Si $i_0, \dots, i_{n+1} \in I$ sont tels que $a\(i_0,i_1\) = \dots = a\(i_n,i_{n+1}\) = 1$ , $z = \(z_1,z_2\)$ est dans un voisinage de $\bar{U_{i_{n+1}}}$ et $w = \(w_1,w_2\)$ est dans un voisinage de $\bar{V_{i_0}}$ alors on a
\begin{equation}\label{rec}
G^{i_0 \dots i_{n+1}}\(z,w\) = \frac{1}{2i \pi} \s{\partial D_1^{i_n}}{\( \frac{1}{2i \pi} \s{\partial D_2^{i_n}}{G^{i_n i_{n+1}} \(z, x\) G^{i_0 \dots i_n}\(x,w\)}{x_2}\) }{x_1}  
.\end{equation}
La preuve de cette formule est en annexe \ref{a2}.

Si $i_0, \dots, i_n \in I$ sont tels que $a\(i_0,i_1\) = \dots = a\(i_{n-1},i_n\) = 1$, on définit un opérateur borné $\L^{i_0 \dots i_n}$ de $\B_{i_0}$ dans $\B_{i_n}$ par 
\begin{equation}\label{def3}
\L^{i_0 \dots i_n} \phi \(z\) = \frac{1}{2i \pi} \s{\partial D_1^i}{ \( \frac{1}{2i \pi} \s{\partial D_2^i}{G^{i_0 \dots i_n}\(z,w\)\phi\(w\)}{w_2}\)}{w_1} 
,\end{equation}
pour $\phi \in \B_{i_0}$ et $z \in U_{i_n}$.

La formule de récurrence \eqref{rec} et le théorème de Fubini assurent que si $a\(i_0,i_1\) = \dots = a\(i_n,i_{n+1}\) = 1$ alors 
\[
\L^{i_n i_{n+1}} \circ \L^{i_0 \dots i_n} = \L^{i_0 \dots i_{n+1}}
.\]
Ainsi si on pose 
\[
\L = \sum_{i,j \in I, a\(i,j\) = 1} b_j \circ \L^{ij} \circ \pi_i
\]
alors pour tout entier naturel $n$ non nul on a 
\begin{equation}\label{puiss}
\L^n = \sum_{ a\(i_0,i_1\) = \dots = a\(i_{n-1},i_n\) = 1 } b_{i_n} \circ \L^{i_0 \dots i_n} \circ \pi_{i_0}
.\end{equation}

Rugh prouve alors dans \cite{R1} que $\L$ est $\frac{2}{3}$-sommable et que pour tout entier naturel $n$ non-nul on a
\begin{equation}\label{tr}
\tr \(\L^n\) = -\sum_{ a\(i_0,i_1\) = \dots = a\(i_{n-1},i_0\) = 1 } \frac{s_{i_0 \dots i_{n-1} i_0} \prod_{k=0}^{n-1} h\(f^{i_0 \dots i_k} \(x^{i_0 \dots i_{n-1} i_0}\)\)}{\det \(Df^{i_0 \dots i_{n-1} i_0}\(x^{i_0 \dots i_{n-1} i_0}\) - I \)}
.\end{equation}
La preuve de ce fait est en annexe \ref{a3}.

\subsection{Perturbation analytique d'applications analytiques réelles hyperboliques}\label{22}

On justifie dans cette partie que la construction de \S \ref{21} reste valable sous de petites perturbations.

\begin{lm}\label{ouv}
Soient $D_1,D_2,D_1',D_2'$ des disques dans $\C$ et $U$ un voisinage de $\bar{D_1} \times \bar{D_2}$ dans $\C^2$. Soit $V$ un voisinage ouvert de $0$ dans $\C^d$, où $d$ est un entier naturel. Pour tout $t$ dans $V$ on se donne une application holomorphe $f_t$ sur $U$ à valeurs dans $\C^2$ de telle manière que l'application $\(t,z\) \mapsto f_t\(z\)$ soit holomorphe sur $V \times U$. On suppose que $f_0$ est analytique réelle hyperbolique de $D_1 \times D_2$ sur $D_1' \times D_2'$ alors il existe un voisinage $W$ de $0$ dans $V$ tel que pour tout $t \in W$ l'application $f_t$ est analytique réelle hyperbolique de $D_1 \times D_2$ sur $D_1' \times D_2'$.
\end{lm}

\begin{proof}
Il est aisé de voir que la première condition dans la définition d'une application analytique réelle hyperbolique est encore vérifiée pour $t$ suffisamment proche de $0$. On se donne $U_1$ et $U_2$ comme dans la seconde condition. On note $r'$ le rayon de $D_2'$ et $x'$ son centre. Soit $K$ un compact de $U_1$ contenant $D_1$. On choisit $r'' > r'$ de telle manière que le disque fermé de $D_2''$ centre $x'$ et de rayon $r''$ soit inclus dans $U_2$. Alors $\phi_{0,s} \( K \times \bar{D_2''} \)$ est un compact inclus dans $D_2$ et $f_{0,2} \(K \times \( \bul{D_2} \setminus \phi_{0,s} \(K \times \bar{D_2''}\) \) \) \subset \C \setminus D_2''$. Par conséquent, $f_{0,2} \( K \times \partial D_2 \) \subset \C \setminus \bul{D_2''}$. Donc pour $t$ suffisamment proche de $0$, $f_{t,2} \( K \times \partial D_2 \)$ ne rencontre pas le disque fermé $\bar{D_2'''}$ de centre $x'$ et de rayon $r''' \in \left]r,r'' \right[$. Alors pour $w_1 \in K$ et $z_2 \in \bar{D_2'''}$, le nombre de solutions $w_2$ dans $\bul{D_2}$ de l'équation $f_{t,2} \(w_1,w_2\) = z_2$ est
\[
\frac{1}{2i \pi} \s{\partial D_2}{\frac{\partial_2 f_{t,2} \(w_1,w_2\)}{f_{t,2}\(w_1,w_2\) - z_2}}{w_2}
\]
qui est continue en $w_1,z_2$ et $t$ et vaut donc $1$ lorsque $t$ est suffisamment proche de $0$.
\end{proof}

Dans la suite on se donne :
\begin{itemize}
\item un ensemble non-vide $I$;
\item une matrice $a : I \times I \to \set{0,1}$;
\item un réel $ \epsilon$ strictement positif que l'on s'autorise à réduire;
\item pour tout $i \in I$ deux disques $D_1^i$ et $D_2^i$ dans $\C$ et un voisinage ouvert $W_i$ de $\bar{D_1^i} \times \bar{D_2^i}$ dans $\C^2$;
\item pour tout $i,j \in I$ tels que $a\(i,j\) = 1$ une application holomorphe $\(t,z\) \in \D{0}{\epsilon} \times W_i \mapsto f_t\(z\)$. 
\end{itemize}
On suppose qu'en fixant $t=0$, les éléments ci-dessus forment un système d'applications analytiques réelles hyperboliques. Par le lemme \ref{ouv}, c'est donc le cas pour tout $t$ quitte à réduire $\epsilon$. On reprend les notations de \S \ref{21} en ajoutant un indice $t$ pour spécifier le système que l'on utilise. On remarquera que les applications ainsi définies sont analytiques en $t$ par le lemme \ref{pf}. En outre, on suppose que lorsque $t$ est réel et $i,j$ sont tels que $a\(i,j\)= 1$, on a $\(f^{ij}_t\)^{-1} \(f^{ij}_t\(W_i\) \cap \R^2\) = W_i \cap \R^2$. Cette condition supplémentaire apparaîtra naturellement lorsqu'on appliquera nos résultats dans la sous-partie suivante.

Si $i,j \in I$ sont tels que $a\(i,j\) = 1$, l'hypothèse ci-dessus implique que lorsque $t$ est réel $\phi_{t,s}^{ij}$ envoie $\(D_1^i \times D_2^j\) \cap \R^2$ dans $\R^2$ et donc $\partial_2 \phi_{t,s}^{ij}$ est à valeurs réelles sur $D_1^j \times D_2^j \cap \R^2$ et ne s'annule pas (car une application holomorphe injective est un difféomorphisme sur son image). On note $s_{ij}$ son signe (qui ne dépend pas de $t$). Avec les notations de \S \ref{21}, si $i_0, \dots, i_n \in I$ sont tels que $a\(i_0,i_1\) = \dots = a\(i_{n-1},i_n\) = 1$ alors $s_{i_0 \dots i_n}$ est le signe de $\partial_2 \phi_{t,s}^{i_0 \dots i_n}$ sur $\(D_1^{i_0} \times D_2^{i_n}\) \cap \R^2$ lorsque $t$ est réel par la formule \eqref{signe} et le lemme \ref{pf}. Si de plus $i_0 = i_n$, alors $x_t^{i_0 \dots i_n} \in \R^2$ (puisqu'il est obtenu par itération d'applications qui préservent $\R^2$) et on a 
\[
\partial_2 f_{t,2}^{i_0 \dots i_n} \(x_t^{i_0 \dots i_n}\) = \( \partial_2 \phi_{t,s}^{i_0 \dots i_n} \(x_t^{i_0 \dots i_n}\)\)^{-1}
\]
et donc le signe de $\partial_2 f_{t,2}^{i_0 \dots i_n} \(x_t^{i_0 \dots i_n}\)$ est $s_{i_0 \dots i_n}$. Le point fixe $x_t^{i_0 \dots i_n}$ étant hyperbolique, les valeurs propres de $d\(f_t^{i_0 \dots i_n}\)\(x_t^{i_0 \dots i_n}\)$ sont réelles et $\tr{D\(f_t^{i_0 \dots i_n}\)\(x_t^{i_0 \dots i_n}\)}$ est du signe de sa valeur de module plus grand que $1$. Or
\[
\tr{D\(f_t^{i_0 \dots i_n}\)\(x_t^{i_0 \dots i_n}\)} = \partial_1 f_{t,1}^{i_0 \dots i_n} \(x_t^{i_0 \dots i_n}\) + \partial_2 f_{t,2}^{i_0 \dots i_n} \(x_t^{i_0 \dots i_n}\)
\]
et le signe de $\tr{D\(f_t^{i_0 \dots i_n}\)\(x_t^{i_0 \dots i_n}\)}$ est donc $s_{i_0 \dots i_n}$. Le signe de $\det \(D\(f_t^{i_0 \dots i_n}\)\(x_t^{i_0 \dots i_n}\) - I \)$ est donc $- s_{i_0 \dots i_n}$.

On se donne une application holomorphe $ \(v,z\) \mapsto h_v \(z\)$ définie sur $\mathcal{U} \times \coprod_{i \in I} W_i$, où $\mathcal{U}$ est un ouvert de $\C^D$, $D$ entier naturel. La partie précédente fournit un Banach $\B$ et une famille d'opérateurs $\( \L_{v,t} \)_{v \in \mathcal{U}, t \in \D{0}{\epsilon}}$. Le point précédent assure que lorsque $t$ est réel pour tout entier naturel $n$ non-nul on a
\begin{equation}\label{trasse}
\tr{\L_{v,t}^n} = \sum_{ a\(i_0,i_1\) = \dots = a\(i_{n-1},i_0\) = 1 } \frac{\prod_{k=0}^{n-1} h_v\(f_t^{i_0 \dots i_k} \(x_t^{i_0 \dots i_{n-1} i_0}\)\)}{\left| \det \(Df_t^{i_0 \dots i_{n-1} i_0}\(x_t^{i_0 \dots i_{n-1} i_0}\) - I \) \right|}
.\end{equation}

Il existe une application holomorphe $\(v,t\) \mapsto L_{v,t} \in \mathcal{B}' \oh \mathcal{B}$, telle que $L_{v,t} $ est un noyau $\frac{2}{3}$-sommable de $\L_{v,t}$. En effet, si $i,j \in I$ sont tels que $a\(i,j\) = 1$ alors la formule \eqref{def3} permet de définir un opérateur $M_{v,t}^{,ij}$ à valeurs dans $\A{X_j}$ où $X_j$ est un voisinage de $\bar{U_j}$ dans $\hat{\C}^2$ qui factorise $M_{v,t}^{ij}$. $\L_{v,t}^{ij}$ dépend de manière analytique de $v$ et de $t$ (pour les mêmes raisons que dans le cas des applications dilatantes du cercle,la continuité est aisée à établir et on en déduit l'analyticité en énonçant un résultat similaire au lemme \ref{cont2}). Or l'injection de $\A{X_j}$ dans $\B_j$ est $\frac{2}{3}$-sommable.

On en déduit que l'application définie sur $\mathcal{U} \times \C \times \D{0}{\epsilon}$
\[
\(z,v,t\) \mapsto \det \( I - z \L_{v,t} \)
\]
est holomorphe. Et la formule \eqref{trasse} nous en donne une expression lorsque $t$ est réel.

\subsection{Applications aux difféomorphismes d'Anosov du tore}\label{23}

On applique dans cette partie les résultats de \S \ref{21} et \S \ref{22} au cas d'un difféomorphisme d'Anosov du tore, suivant \cite{R2}. Le résultat essentiel de cette partie est le théorème \ref{gros}.

Fixons quelques notations. On identifie le fibré tangent de $\T^2$ avec $\T^2 \times \R^2$ et son complexifié avec $\T^2 \times \C^2$. On note $| . |$ la norme hilbertienne usuelle sur $\C^2$. On note $\pi$ la projection de $\C^2$ sur $\C^2 / \Z^2$ et $\T^2 = \R^2 / \Z^2 \simeq \pi \(\R^2\)$. On se donne $\epsilon > 0$, que l'on s'autorise à réduire autant que nécessaire, et pour tout $t \in \left] - \epsilon, \epsilon \right[$ un difféomorphisme d'Anosov analytique et transitif $f_t$ de $\T^2$. On suppose en outre que l'application $\(t,x\) \mapsto f_t\(x\)$ est analytique sur $\left]-\epsilon,\epsilon \right[ \times \T^2$ et on note $f = f_0$. On note $\(t,x\) \in  \mapsto F_t\(x\)$ un relevé défini sur $\left]-\epsilon,\epsilon \right[ \times \R^2$ de l'application $\(t,x\) \mapsto f_t\(x\)$ définie sur $\left]-\epsilon,\epsilon \right[ \times \T^2$ et on pose $F = F_0$. Par stabilité structurelle de $f$, quitte à réduire $\epsilon$, on peut supposer que les $f_t$ pour $t$ réel sont tous topologiquement conjugués.

Dire que $f$ est un difféomorphisme d'Anosov de $\T^2$, c'est dire qu'il existe des constantes $C$ et $\lambda < 1$ et pour tout $x \in \T^2$ deux droites supplémentaires $E_x^u$ et $E_x^s$ dans $T_x \T^2 \simeq \R^2$ telles que :
\begin{itemize}
\item pour tout $x \in \T^2$, on a $T_x f \(E_x^u\) = E_{f\(x\)}^u$ et $T_x f \(E_x^s\) = E_{f\(x\)}^s$;
\item pour tout $x \in \T^2$ et tout entier naturel $n$, on a $\| \left. T_x f^n \right|_{E_x^s} \| \leqslant C \lambda^n$ et $\| \left. T_x f^{-n} \right|_{E_x^u} \| \leqslant C \lambda^n$.
.\end{itemize}

L'application $\(t,x\) \mapsto f_t\(x\)$ étant analytique sur $\left]-\epsilon,\epsilon \right[ \times \T^2$, elle s'étend, quitte à réduire $\epsilon$ en une application holomorphe, que l'on note de la même manière ainsi que son relevé, sur $\D{0}{\epsilon} \times U$ où $U$ est un voisinage de $\T^2$ dans $\C^2 / \Z^2$.

On se donne une application analytique $g : \T^2 \to \R$ qui s'étend en une application holomorphe de $U$ dans $\C$, quitte à réduire $U$. Nous allons démontrer :
\begin{thm}\label{gros}
Sous ces hypothèses et quitte à réduire $\epsilon$, pour tout $v \in \R$ et tout $t \in \left] - \epsilon, \epsilon \right[$ la série entière 
\[
\sum_{n > 0} \frac{1}{n} \( \sum_{x \in \fix{f_t^n}} \frac{\exp \(v \sum_{j=0}^{n-1} g\(f_t^j \(x\)\)\)}{\left| \det \( D \( f_t^n \) \( x \) - I \) \right|}  \) z^n
\]
a un rayon de convergence strictement positif et l'application
\[
z \mapsto \exp\( - \sum_{n > 0} \frac{1}{n} \( \sum_{x \in \fix{f_t^n}} \frac{\exp \(v \sum_{j=0}^{n-1} g\(f_t^j \(x\)\)\)}{\left| \det \( D \( f_t^n \) \( x \) - I \) \right|}  \) z^n \)
\]
s'étend en une application entière $z \mapsto d\(z,v,t\)$ qui a un zéro simple en $\exp\(-P_t\(\phi_t^u + vg\) \)$, où $P_t$ désigne la pression topologique pour la dynamique de $f_t$ et $\phi_t^u : x \mapsto - \log\( \left| \left. \det Df_t\(x\) \right|_{E_{t,x}^u}\right|\)$ (où $E_{t,x}^u$ désigne la direction instable de $f_t$ au point $x$).

De plus, l'application $d$ est analytique en ses trois arguments.
\end{thm}

\begin{rmq}\label{adaptation}
On en déduit une formule pour la réponse linéaire en fonction des points périodiques de $f$ de la même manière que dans le cas du cercle. Plus précisément, la proposition \ref{expr} et son corollaire \ref{derivee} restent vrais sans modification des preuves lorsqu'on remplace $\mu_t$ par la mesure physique de $f_t$ (c'est-à-dire la mesure d'équilibre de $\phi_t^u$, voir pour cela le corollaire 4.13 de \cite{Bow2}) et l'application $d$ définie dans la proposition \ref{ppal} par celle définie dans le théorème \ref{gros} (on a $P_t\(\phi_t^u\) = 0$ par le théorème 4.11 de \cite{Bow2})
\end{rmq}

\begin{lm}\label{233}
Si $\epsilon$ et $U$ sont choisis suffisamment petit, alors pour $t$ réel et $x \in U$ on a $f_t\(x\) \in \T^2$ si et seulement si $x \in \T^2$.
\end{lm}

\begin{proof}
Pour $t$ réel, la différentielle de $f_t$ en un point du tore est \og à coefficients réels \fg{} et inversible. Un raisonnement similaire à celui de la démonstration du lemme \ref{tech} permet donc de conclure.
\end{proof}

On supposera dans la suite que les hypothèses du lemme \ref{233} sont vérifiées. Pour tout $x \in \T^2$ on note $p_x^u$ la projections sur le $\C$-sous-espace vectoriel de $\C^2$ engendré par $E_x^u$ (noté $F_x^u$) selon celui engendré par $E_x^s$ (noté $F_x^s$) et $p_x^s = I -p_x^u$. On choisit $ \theta \in \left]\lambda,1 \right[$ et on pose pour tout $x \in \T^2$ et tout $v \in \C^2$:
\[
\|v\|_x^s = \sqrt{\sum_{k \geqslant 0} \theta^{-2k} \left| T_x f^k \circ p_x^s \(v\) \right|^2}, \quad
\|v\|_x^u = \sqrt{\sum_{k \geqslant 0} \theta^{-2k} \left| T_x f^{-k} \circ p_x^u \(v\) \right|^2}
\]
et
\[
\|v\|_x = \max \( \|v\|_x^s,\|v\|_x^u\)
.\] 
Cette norme dont l'introduction est due à Mather fait de la différentielle de $f$ une contraction (resp. une dilatation) lorsqu'on la restreint aux directions stables (resp. instables).

Si $x \in \R^2$ et $\delta > 0$ on pose $R_x \( \delta \) = \set{y \in \C^2 : \|y-x\|_x \leqslant \delta }$ où on a noté $\|.\|_x = \|.\|_{\pi\(x\)}$ (on utilisera la même convention pour les projections et les sous-espaces de $\C^2$). Remarquons que si $\delta < \frac{1}{4}$ alors $\pi$ induit un difféomorphisme d'un voisinage de $R_x \( \delta \)$ sur son image. De plus, si on pose 
\[
D_1^x \(\delta\) = \set{ y \in F_x^s : \|y-x\|_x^s \leqslant \delta}
\textrm{ et }
D_2^x \(\delta\) = \set{ y \in F_x^u : \|y-x\|_x^u \leqslant \delta}
,\]
le rectangle $R_x \( \delta\)$ s'écrit naturellement comme un produit de deux disques :
\[
R_x \( \delta \) = D_1^x \( \delta \) + D_2^x \( \delta \)
.\]

Rugh prouve alors dans \cite{R2} le résultat suivant.

\begin{lm}
Il existe $\gamma > 0$ et $\delta >0$ arbitrairement petits tels que si $x,y \in \T^2$ vérifient $|x-y| < \gamma $ et $|F\(x\) - F\(y\)| < \gamma$ alors $F$ induit une application analytique réelle hyperbolique de $R_x\(\delta\) = D_1^x \(\delta\) + D_2^x \(\delta\)$ sur $R_{F\(y\)} \(\delta\) =  D_1^{F\(y\)} \(\delta\) + D_2^{F\(y\)} \(\delta\)$.
\end{lm}

\begin{rmq}
Si on voulait se ramener exactement au cadre des parties précédentes, il faudrait considérer l'application $\tilde{F}$ vérifiant
\[
\tilde{F} \circ \(p_x^s,p_x^u\) = \(p_{F\(y\)}^s,p_{F\(y\)}^u\) \circ F 
\]
et dire que cette application est analytique réelle hyperbolique. C'est ce qu'on entend dans l'énoncé du lemme.
\end{rmq}

On se donne de tels $\delta$ et $\gamma$ avec $\delta < \frac{1}{2}$. Dans la suite, on se donne une partition de Markov $\(R_i\)_{i \in I}$ de $\T^2$ pour $f$ en des rectangles de diamètres plus petits que $\min\( \frac{\gamma}{ \Delta +1}, \frac{\delta}{M} \)$ et assez petits pour que les résultats de \cite{Manning} s'appliquent (voir \cite{Bow2} ou \cite{Bow1} pour la définition et la construction des partitions de Markov).

Afin de démontrer le théorème \ref{gros}, nous rappelons maintenant la construction de \cite{Manning} qui permet de rectifier l'erreur dans le décompte des points périodiques d'un difféomorphisme Anosov lorsqu'on a recours à la dynamique symbolique. On dit que des rectangles $R_{i_1}, \dots, R_{i_r}$ sont reliés si $ \cap_{j=1}^r R_{i_j} \neq \emptyset$. On note $q$ le plus grand entier tel qu'il existe $q$ rectangles distincts reliés. Si $m = \(m_1, \dots , m_k \)$ est un uplet d'entiers non-nuls tels que $|m| = \sum_{j=1}^k m_j \leqslant q$, on note $\mathcal{C}_m$ l'ensemble des $k$-uplets $\(e_1, \dots, e_k\)$, où pour $j = 1, \dots ,k$, $e_j$ est un ensemble de $m_j$ rectangles, et les éléments de $e_1, \dots, e_k$ sont $|m|$ rectangles reliés distincts. 

On définit une matrice de transition $a_m$ sur $\mathcal{C}_m$ en posant si $e = \(e_1, \dots, e_k\)$ et $f = \(f_1, \dots ,f_k\)$  sont des éléments de $\mathcal{C}_m$, $a_m\(e,f\)=1$ si et seulement si pour $j=1, \dots ,k$ il existe une bijection de $e_j$ sur $f_j$ telle que pour tout $R \in e_j$, $f \( \bul{R} \) $ rencontre $\bul{\phi\(R\)}$ (notons qu'une telle bijection est alors unique, voir \cite{Manning}). On note $\(C_m,\sigma_m\)$ le sous-décalage de type fini ainsi défini. Notons que pour $k=1$ et $m=\(1\)$, on retrouve le sous-décalage de type fini habituellement associé à $f$.

Il existe alors une application continue $\pi_m : C_m \to \T^2$ qui semi-conjugue $\sigma_m$ et $f$ et telle que si $e=\(e^n\)_{n \in \Z}$ est un élément de $C_m$ et $e^0 = \(e_1, \dots, e_k\)$ alors $ \pi_m \(e\) \in \cap_{j=1}^k \(\cap e_j\)$ (voir \cite{Manning}).

Si $n$ est un entier naturel non-nul et $x \in \fix{f^n}$, on note $N_m^n\(x\) $ le nombre d'antécédents de $x$ par $\pi_m$ qui sont dans $\fix{\sigma_m^n}$.

Pour $k=1,\dots ,q$, on note $G_k$ l'ensemble des $k$-uplets $m = \(m_1, \dots , m_k \)$ d'entiers non-nuls tels que $|m| = \sum_{j=1}^k m_j \leqslant q$. Manning prouve alors dans \cite{Manning} que pour tout entier naturel non-nul $n$ et tout $x \in \fix{f^n}$ :
\begin{equation}\label{comptage}
\sum_{k=1}^q \sum_{m \in G_k} \(-1\)^{k+1}  N_m^n \(x\) = 1
.\end{equation}

Fixons temporairement $k \in \set{1,\dots,q}$ et $m \in G_k$. Pour tout $e = \(e_1, \dots ,e_k\) \in \mathcal{C}_m$, on choisit $x_e \in \cap_{j=1}^k \( \cap e_j\)$ (un tel point existe bien par définition de $\mathcal{C}_m$) et un relevé $\tilde{x_e}$ de $x_e$. Remarquons qu'alors $\pi\(R_{\tilde{x_e}} \(\delta\)\)$ recouvre $x_e \in \cap_{j=1}^k \( \cap e_j\)$. 

Si $e= \(e_1,\dots,e_k\)$ et $f = \(f_1,\dots,f_k\)$ sont des éléments de $\mathcal{C}_m$ tels que $a_m\(e,f\) = 1$ alors il existe $R \in e_1$ et $R' \in f_1$ tels que $f\( \bul{R} \)$ rencontre $\bul{R'}$. On note $y$ un point dans l'intersection. Alors $y$ a un relevé $\tilde{y}$ tel que $|\tilde{y}-\tilde{x_e}| \leqslant \textrm{diam } R \leqslant \frac{\gamma}{\Delta + 1}$ et il existe $\tau_{ef} \in \Z^2$ tel que de la même manière $|\tilde{x_f} - \(F\(\tilde{y}\) - \tau_{ef}\)| \leqslant \frac{\gamma}{\Delta + 1}$. Par conséquent
\begin{align*}
\left| \tilde{x_e} - F^{-1} \(\tilde{x_f} + \tau_{ef}\) \right| & \leqslant \left|\tilde{x_e}-\tilde{y} \right| + \left|F^{-1} \(F\(\tilde{y}\) \) - F^{-1}\(\tilde{x_f} + \tau_{ef} \)\right| \\
    & \leqslant \frac{\gamma}{\Delta +1} + \Delta \left| \tilde{y} - \(\tilde{x_f} + \tau_{ef} \)\right| \leqslant \frac{\gamma}{\Delta +1} + \Delta \frac{\gamma}{\Delta + 1} \leqslant \gamma. 
\end{align*}
Et de même 
\[
\left| F \(\tilde{x_e}\) - \(\tilde{x_f} + \tau_{ef}\) \right| \leqslant \gamma
.\]
Et donc $F$ est hyperbolique de $R_{\tilde{x_e}} \(\delta\)$ sur $R_{\tilde{x_f} + \tau_{ef}} \(\delta\)$. Autrement dit $f$ induit une application analytique réelle hyperbolique de $\pi\(R_{\tilde{x_e}} \(\delta\)\)$ sur $\pi\(R_{\tilde{x_f}} \(\delta\)\)$ (qui sont biholomorphes à des produits de disques). 

On a donc construit un système d'applications analytiques réelle hyperboliques. On applique alors la construction de \S \ref{22} (ce qui implique potentiellement de réduire $\epsilon$, mais seulement un nombre fini de fois) avec $h_v \(x\) = \exp\(v g\(x\)\)$ pour $v \in \C$ et $ x \in U$. On note $\L_{m,v,t}$ l'opérateur obtenu et $\B_m$ le Banach sur lequel il agit. Nous allons alors montrer :
\begin{lm}\label{Man}
Pour tout $t \in \left]- \epsilon, \epsilon\right[$, tout $v \in \C$ et tout entier naturel non-nul $n$, on a 
\begin{equation}\label{salt}
\sum_{k=1}^q \sum_{m \in G_k} \(-1\)^{k+1} \tr \L_{m,v,t}^n = \sum_{x \in \fix{f_t^n}} \frac{\exp \(v \sum_{j=0}^{n-1} g\(f_t^j \(x\)\)\)}{\left| \det \( D \( f_t^n \) \( x \) - I \) \right|} 
.\end{equation}
\end{lm}
\begin{proof}
Remarquons que les deux membres de cette équation sont analytiques (réels) en $t$ : celui de gauche est la restriction d'une application holomorphe par la partie précédente, celui de droite par le théorème des fonctions implicites et le fait que les $f_t$ sont tous conjugués. Il suffit donc de montrer que l'égalité a lieu dans un voisinage de $0$ (qui peut éventuellement dépendre de $n$). 

Fixons donc $n$ et montrons cette égalité dans un voisinage de $0$. Par le théorème des fonctions implicites, si $x \in \fix{f_n}$ ($x$ est hyperbolique pour $f^n$), il existe une application continue $t \mapsto c_x\(t\)$ définie sur un voisinage de $0$ telle que $c_x\(0\) = x$ et pour tout $t$ où cette application est définie $c_x\(t\) \in \fix{f_t^n}$. De plus, si l'ensemble de définition de $c_x$ est choisi suffisamment petit, une autre application avec ces propriétés coïncide avec $c_x$ sur la composante connexe de $0$ dans l'intersection de son domaine de définition avec celui de $c_x$. On se donne $\epsilon' >0$ tel que pour tout $x \in \fix{f_n}$ l'application $c_x$ est définie sur $\left]- \epsilon', \epsilon'\right[$ avec la propriété d'unicité ci-dessus et pour tout $t \in \left] - \epsilon',\epsilon' \right[$, les $c_x\(t\)$ pour $x \in \fix{f^n}$ sont distincts (ceci est possible car par le théorème d'inversion locale les points $n$ périodiques de $f$ sont isolés et donc en nombre fini). Mais alors les $c_x\(t\)$ pour $x \in \fix{f^n}$ sont exactement les éléments de $\fix{f_t^n}$ car $f_t$ est topologiquement conjugué à $f$ et a donc autant de points $n$-périodiques. 

Ainsi si $k \in \set{1, \dots, q}$, $m \in G_k$ et $e_0, \dots ,e_{n-1}$ sont des éléments de $\mathcal{C}_m$ tels que $a_m \(e_0,e_1\) = \dots = a_m\(e_{n-2},e_{n-1}\) = a_m\(e_{n-1},e_0\) = 1$ alors par la propriété d'unicité de $c_{x^{e_0 \dots e_{n-1} e_0}}$ (avec les notations des parties précédentes) on a pour tout $t \in \left] - \epsilon', \epsilon' \right[$ :
\[
x_t^{e_0 \dots e_{n-1} e_0} = c_{x^{e_0 \dots e_{n-1} e_0}} \(t\)
.\]
Or si on note de manière légèrement abusive$\(e_0 \dots e_{n-1} \)^{\omega}$ le point périodique de $\sigma_m$ associé au cycle $e_0, \dots, e_{n-1}, e_0$, le point $\pi_m \( \(e_0 \dots e_{n-1} \)^{\omega}\)$ vérifie la propriété qui caractérise $x^{e_0 \dots e_{n-1} e_0}$ donnée dans la première partie (c'est un point périodique de $f$ dont les itérés sont dans les bons rectangles) et donc 
\[
\pi_m \( \(e_0 \dots e_{n-1} \)^{\omega}\) = x^{e_0 \dots e_{n-1} e_0}
\]
puis pour tout $t \in \left] - \epsilon', \epsilon' \right[$ :
\[
x_t^{e_0 \dots e_{n-1} e_0} = c_{\pi_m \( \(e_0 \dots e_{n-1} \)^{\omega}\)} \(t\)
.\]

La formule \eqref{trasse} pour la trace de l'opérateur $\L_{m,v,t}^n$ donne alors pour $t \in \left] - \epsilon', \epsilon' \right[$ :
\begin{align*}
\tr{\L_{m,v,t}^n} & = \sum_{a_m\(e_0,e_1\) = \dots = a_m\(e_{n-1},e_0\) = 1 } \frac{\prod_{k=0}^{n-1} h_v\(f_t^n \(x_t^{e_0 \dots e_{n-1} e_0}\)\)}{\left| \det \(Df_t^n \(x_t^{e_0 \dots e_{n-1} e_0}\) - I \) \right|} \\
     & = \sum_{a_m\(e_0,e_1\) = \dots = a_m\(e_{n-1},e_0\) = 1 } \frac{\exp \( v \sum_{j=0}^{n-1} g \(f_t^n \( c_{\pi_m \( \(e_0 \dots e_{n-1} \)^{\omega}\)} \(t\) \) \)\)}{\left| \det \(Df_t^n \(c_{\pi_m \( \(e_0 \dots e_{n-1} \)^{\omega}\)} \(t\)\) - I \) \right|} \\
     & = \sum_{y \in \fix{\sigma_m^n}} \frac{\exp \( v \sum_{j=0}^{n-1} g \(f_t^n \( c_{\pi_m \( y\)} \(t\) \) \)\)}{\left| \det \(Df_t^n \(c_{\pi_m \( y\)} \(t\)\) - I \) \right|} \\
     & = \sum_{x \in \fix{f^n}} N_m^n\(x\) \frac{\exp \( v \sum_{j=0}^{n-1} g \(f_t^n \( c_x \(t\) \) \)\)}{\left| \det \(Df_t^n \(c_x \(t\)\) - I \) \right|}.
\end{align*}

On obtient alors pour $t \in \left] - \epsilon', \epsilon' \right[$ :
\begin{align*}
\sum_{k=1}^q \sum_{m \in G_k} \(-1\)^{k+1} \tr \L_{m,v,t}^n & = \sum_{x \in \fix{f_t^n}}\( \(\sum_{k=1}^q \sum_{m \in G_k} \(-1\)^{k+1}  N_m^n \(x\) \) \frac{\exp \(v \sum_{j=0}^{n-1} g\(f_t^j \(x\)\)\)}{\left| \det \( D \( f_t^n \) \( x \) - I \) \right|} \) \\
 & = \sum_{x \in \fix{f_t^n}} \frac{\exp \(v \sum_{j=0}^{n-1} g\(f_t^j \(x\)\)\)}{\left| \det \( D \( f_t^n \) \( x \) - I \) \right|} 
\end{align*}
par la formule \eqref{comptage}, ce qui achève la démonstration de la formule \eqref{salt}.
\end{proof}

\begin{proof}[Démonstration du théorème \ref{gros}]
Du lemme, \ref{Man}, on déduit que pour tout $t \in \left]- \epsilon, \epsilon \right[$ et $v \in \R$ la série entière
\[
\sum_{n > 0} \frac{1}{n} \( \sum_{x \in \fix{f_t^n}} \frac{\exp \(v \sum_{j=0}^{n-1} g\(f_t^j \(x\)\)\)}{\left| \det \( D \( f_t^n \) \( x \) - I \) \right|}  \) z^n
\]
a un rayon de convergence strictement positif et que l'application
\[
z \mapsto \exp\( - \sum_{n > 0} \frac{1}{n} \( \sum_{x \in \fix{f_t^n}} \frac{\exp \(v \sum_{j=0}^{n-1} g\(f_t^j \(x\)\)\)}{\left| \det \( D \( f_t^n \) \( x \) - I \) \right|}  \) z^n \)
\]
s'étend à $\C$ en une application méromorphe 
\[
z \mapsto d\(z,v,t\) = \prod_{k=1}^q \prod_{m \in G_k} \det\(I - z \L_{m,v,t}\)^{\(-1\)^{k+1}} 
.\]
Dans \cite{R2}, Rugh prouve que cette application est en fait entière (une esquisse de preuve de ce fait est l'objet de \S \ref{24}).

Montrons que $d$ est analytique en ses trois arguments (c'est en fait la restriction d'une application holomorphe). La fonction $d$ s'écrit comme le quotient de la restriction de deux applications $d_1$ et $d_2$ qui d'après la partie précédente sont holomorphes sur $\C \times \C \times \D{0}{\epsilon}$. Soit $\(z_0,v_0,t_0\) \in \C \times \R \times \left] - \epsilon, \epsilon \right[$. Il existe alors $R > 0$ tel que $d_2\(.,v_0,t_0\)$ ne s'annule pas sur le cercle $S$ de centre $z_0$ et de rayon $R$ ($d_2\(.,v_0,t_0\)$ n'est pas identiquement nulle car elle vaut $1$ en $0$). Alors pour $v$ et $t$ dans un voisinage (respectivement dans $\C$ et dans $ \D{0}{\epsilon}$) de $v_0$ et $t_0$, $d_2\(.,v,t\)$ reste hors d'un voisinage de $0$ sur ce cercle et on peut donc écrire pour de tels $v$ et $t$ et $z$ dans le disque ouvert de centre $z_0$ et de rayon $R$ (on utilise ici l'entièreté de $d\(.,v.t\)$):
\[
d\(z,v,t\) = \frac{1}{2i \pi} \s{ \partial D}{ \frac{d_1\(w,v,t\)}{\(w-z\)d_2\(w,v,t\)}}{w}
.\]
Et il est classique qu'une telle expression est holomorphe en $\(z,v,t\)$. L'application $d$ est donc analytique en ses trois arguments au voisinage de $\(z_0,v_0,t_0\)$ puis analytique en ses trois arguments.

Enfin, on va montrer que si $t \in \left] - \epsilon, \epsilon \right[$ et $v \in \R$ alors il existe $R > \exp\(-P_t\(\phi_t^u + vg\)\)$ (où $\phi_t^u : x \mapsto - \log\( \left| \det \left. Df_t\(x\) \right|_{E_{t,x}^u} \right|\)$) tel que $d\(.,v,t\)$ a un unique zéro dans le cercle de centre $0$ et de rayon $R$, que celui-ci est simple et qu'il s'agit de $\exp\(-P_t\(\phi_t^u + vg\)\)$. Sans restriction de généralité, on peut pour cela supposer $t = 0$ et on pose alors $\phi^u = \phi_0^u$ et $P=P_0$.

D'après Ruelle (paragraphes 7.23 et 7.24 de \cite{Ruelle2}) la série entière 
\[
\sum_{n>0} \frac{1}{n} \(\sum_{x \in \fix{f^n}} \exp \( \sum_{j=0}^{k-1} \( \phi^u \(f^j\(x\)\) + v g \(f^j \(x\) \) \)\)\) z^n
\]
a pour rayon de convergence $\exp\(-P\(\phi^u + vg\)\)$ et l'application
\[
z\mapsto \exp \( \sum_{n>0} \frac{1}{n} \(\sum_{x \in \fix{f^n}} \exp \( \sum_{j=0}^{k-1} \( \phi^u \(f^j\(x\)\) + v g \(f^j \(x\) \) \)\)\) z^n \)
\]
s'étend en une application méromorphe $\zeta$ sur un disque de centre $0$ et de rayon $r > \exp\(-P\(\phi^u + vg\)\)$ qui ne s'annule pas et qui a un unique pôle en $ \exp\(-P\(\phi^u + vg\)\)$ qui est simple. Mais alors pour $z$ de module suffisamment petit on a 
\begin{align*}
& d\(z,v,0\) \zeta \(z\) =  \\  & \exp\( \sum_{n>0} \frac{1}{n} \(\sum_{x \in \fix{f^n}} \exp\( v \sum_{j=0}^{k-1} g \(f^j\(x\)\)\) \( \frac{1}{\left|\det\( \left. Df^n \(x\)\right|_{E_x^u}\) \right|} - \frac{1}{\left| \det\(Df^n\(x\) -I\)\right|}\)    \)z^n \) 
.\end{align*}
Si $n$ est un entier naturel et $x \in \fix{f^n}$ on note $a_u\(n,x\)$ et $a_s\(n,x\)$ les valeurs propres de $Df^n\(x\)$ de valeurs absolues respectivement plus petite et plus grande que $1$. L'hypothèse d'hyperbolicité faite sur $f$ assure alors que $|a_u\(n,x\)| \geqslant \lambda^n$ et $|a_s\(n,x\)| \leqslant \lambda^{-n}$.  On a alors
\begin{align*}
\left| \frac{1}{\left|\det\( \left. Df^n \(x\)\right|_{E_x^u}\) \right|} - \frac{1}{\left| \det\(Df^n\(x\) -I\)\right|} \right| & = \left| \frac{1}{\left|a_u\(n,x\)\right|} - \frac{1}{\left|a_u\(n,x\) - 1\right| \left|a_s\(n,x\) - 1\right|}\right| \\
   & \leqslant \frac{1}{|a_u\(n,x\)|} \frac{\left| 1 - \(1- \frac{1}{a_u\(n,x\)}\)\(1-a_s\(n,x\)\)\right|}{\(1- \frac{1}{a_u\(n,x\)}\)\(1-a_s\(n,x\)\)} \\
   & \leqslant \frac{1}{\left|\det\( \left. Df^n \(x\)\right|_{E_x^u}\) \right|} 3 \( \frac{\lambda}{\lambda -1}\)^2 \lambda^{-n}
.\end{align*}
Et donc pour tout entier naturel non-nul $n$, on a 
\begin{align*}
\left| \frac{1}{n} \(\sum_{x \in \fix{f^n}} \exp\( v \sum_{j=0}^{k-1} g \(f^j\(x\)\)\) \( \frac{1}{\left|\det\( \left. Df^n \(x\)\right|_{E_x^u}\) \right|} - \frac{1}{\left| \det\(Df^n\(x\) -I\)\right|}\) \)\right| \\ \leqslant 3 \( \frac{\lambda}{\lambda-1}\)^2 \lambda^{-n} \frac{1}{n} \(\sum_{x \in \fix{f^n}} \exp \( \sum_{j=0}^{k-1} \( \phi^u \(f^j\(x\)\) + v g \(f^j \(x\) \) \)\)\)
.\end{align*}
Notons que ceci fonctionne parce que le poids $\exp\(v g\)$ est positif. On en déduit que la série
\[
\sum_{n>0} \frac{1}{n} \(\sum_{x \in \fix{f^n}} \exp\( v \sum_{j=0}^{k-1} g \(f^j\(x\)\)\) \( \frac{1}{\left|\det\( \left. Df^n \(x\)\right|_{E_x^u}\) \right|} - \frac{1}{\left| \det\(Df^n\(x\) -I\)\right|}\)    \)z^n 
\]
a un rayon de convergence supérieur ou égal à $\lambda \exp\(-P\(\phi^u + vg\) \)$. Par principe du prolongement analytique, l'application $z \mapsto d\(z,v,0\) \zeta\(z\)$ est holomorphe et ne s'annule pas sur le disque de centre $0$ et de rayon $R = \min\(r,\lambda \exp\(-P\(\phi^u + vg\) \)\)$. On en déduit que $\exp\(-P\(\phi^u + vg\) \)$ est l'unique $0$ de $z \mapsto d\(z,v,0\)$ dans le disque de centre $0$ et de rayon $R$ et qu'il est simple. Ce qui achève la preuve du théorème \ref{gros}, à l'entièreté de la fonction $d$ près qui est l'objet de \S \ref{24}.
\end{proof}

\subsection{Esquisse de preuve de l'entièreté du déterminant.}\label{24}

S'inspirant des idées de \cite{R2}, on donne une esquisse de preuve de l'entièreté des déterminants définis dans la partie précédente. On en profite pour donner la preuve d'un résultat légèrement plus fort dans la remarque \ref{unpeuplus} que nous utiliserons pour démontrer la proposition \ref{vitconv2} de \S \ref{25}.

On se donne un difféomorphisme Anosov $f$ de $\T^2$ et une application analytique $g$ de $\T^2$ dans $\R$. Par la même méthode que dans la partie précédente, on montrerait que l'application
\[
z \mapsto \exp \(-\sum_{n > 0} \frac{1}{n} \sum_{x \in \fix{f^n}} \frac{\prod_{k=0}^{n-1} g\(f^k\(x\)\)}{\left| \det \(D f^n \(x\) - I \) \right|}\) 
\]
s'étend en une application méromorphe sur $\C$. On va montrer que cette extension est en fait entière. 

Pour cela on se donne deux orbites périodiques $\mathcal{O}_1$ et $\mathcal{O}_2$ telles que pour chacune il existe une partition de Markov de diamètre suffisamment petit pour appliquer les constructions de la partie précédente telle que les seuls points périodiques inclus dans les bords des rectangles sont ceux de l'orbite en question (pour la construction voir \cite{R2} qui reprend la construction de partitions de Markov par Bowen, voir \cite{Bow2}). On se donne une application analytique $c$ de $\T^2$ dans $\R_+^*$ telle que :
\begin{itemize}
\item $\(s,x\) \in \R \times \R^2 \mapsto c\(x\)^s$ est la restriction d'une application holomorphe sur $\C \times \( \C^2 / \Z^2 \)$;
\item $\prod_{x \in \mathcal{O}_1} c\(x\) = 1$;
\item $\prod_{x \in \mathcal{O}_2} c\(x\) < 1$. 
\end{itemize}
La première hypothèse nous permet d'appliquer la construction de la partie précédente avec  comme poids
\[
h_s\(x\) = g\(x\) c\(x\)^s
\] 
pour $x \in \T^2$ et $s \in \R$ sans causer de problèmes de dépendance logique. Pour construire une telle application, on peut commencer par trouver une application $\tilde{c}$ holomorphe sur $\C^2 / \Z^2$ prenant des valeurs réelles sur $\T^2$, strictement négative sur $\mathcal{O}_2$ et strictement positive sur $\mathcal{O}_1$ (par exemple en utilisant le théorème de Stone-Weierstrass). On obtient alors l'application désirée en renormalisant $e^{\tilde{c}}$.

On note $\(z,s\) \mapsto d\(z,s\)$ l'application analytique ainsi obtenue. Chacune des deux partitions de Markov est à l'origine d'une écriture de $d$ :
\[
d\(z,s\) = \frac{\phi_i\(z,s\)}{\psi_i\(z,s\)}
\]
pour $i = 1,2$ où pour tout $s \in \R$ les applications $\phi_i\(.,s\)$ et $\psi_i\(.,s\)$. De plus, $\psi_i$ est un produit fini d'applications qui au voisinage de l'origine sont de la forme
\[
\(z,s\) \mapsto \exp \( - \sum_{n > 0} \frac{G_i^{kn}C_i^{skn}}{J_{i,n}} z^n \) 
\]
où 
\[
G_i = \prod_{x \in \mathcal{O}_i} g\(x\)
, \quad
C_i = \prod_{x  \in \mathcal{O}_i} c\(x\)
\textrm{ et }
J_{i,n} = \left| \det\(D f^n\(x\) - I\) \right|
\]
pour $x$ un élément quelconque de $\mathcal{O}_i$.

Ainsi $C_1=1$, et donc pour tout $s \in \R$ on a $\psi_1\(.,v\) = \psi_1\(.,0\)$. Par contre, $0 < C_2 < 1$ et donc les séries entières dont les sommes apparaissent dans les exponentielles ci-dessus tendent vers $+ \infty$ lorsque $s$ tend vers $+ \infty$. Par conséquent pour tout $z \in \C$ ,$\psi_2\(z,s\) \neq 0$ pour $s$ suffisamment grand.

Mais alors si $\psi_1\(.,0\)$ a une racine d'ordre $n$ en $z \in \C$ alors pour tout $s$ suffisamment grand $\psi_2\(z,s\) \neq 0$ et donc $\phi_1\(.,v\)$ a une racine d'ordre au moins $n$ en $z$ (puisque $\psi_1\(.,v\) = \psi_1\(.,0\)$ a une racine d'ordre $n$ en $z$). et donc $\phi_1^{(j)}\(z,s\) = 0$ pour $j = 0, \dots, n-1$. Comme ceci a lieu pour tout $s$ suffisamment grand et que $\phi_2$ est analytique en $s$ on a donc également $\phi_1^{(j)}\(z,0\) = 0$ pour $j = 0, \dots , n-1$, $\phi_1\(.,0\)$ a une racine d'ordre au moins $n$ en $z$ et donc $d\(.,0\) $ est holomorphe en $z$.

\begin{rmq}\label{unpeuplus}
Notons que si l'on fait la construction de la partie précédente en remplaçant le poids $\exp\(vg\(x\)\)$ par $\exp\(vg\(x\)+ s a\(x\)\)$  où $a$ est une fonction holomorphe sur $\C^2 / \Z^2$ telle que pour un certain $t \in \D{0}{\epsilon}$ suffisamment petit on ait 
\[
\sum_{x \in \mathcal{O}_{1,t}} a\(x\) = 0
\textrm{ et }
\Re \( \sum_{x \in \mathcal{O}_{2,t}} a\(x\) \) < 0
\]
où $\mathcal{O}_{1,t}$ et $\mathcal{O}_{2,t}$ sont les orbites périodiques de $f_t$ obtenues en suivant $\mathcal{O}_1$ et $\mathcal{O}_2$ par le théorème des fonctions implicites (ce que l'on peut toujours faire si $t$ est suffisamment petit), alors on peut appliquer l'argument ci-dessus pour montrer que $z \mapsto d\(z,v,t\)$ est entière (on ne l'avait montré a priori que pour $t$ réel).
\end{rmq}

\subsection{Vitesse de convergence}\label{25}

\'Ecrivons la fonction $d$ du théorème \ref{gros} comme
\begin{equation}\label{develo}
d\(z,v,t\) = 1 + \sum_{n \leqslant 1} a_n\(v,t\)z^n
.\end{equation}
On va donner une estimation de la vitesse de décroissance des $a_n\(v,t\)$ et de leurs dérivées pour $v$ et $t$ au voisinage de $0$. Par le théorème \ref{gros} et la remarque \ref{unpeuplus} pour $v$ et $t$ suffisamment proche de $0$ (dans $\C$) l'application $z \mapsto d\(z,v,t\)$ est entière. De plus, elle s'écrit 
\[
d\(z,v,t\) = \prod_{k=1}^q \prod_{m \in G_k} \det\(I - z \L_{m,v,t}\)^{\(-1\)^{k+1}} 
\]
avec les notations de \S \ref{23}. La méthode sera la même que dans le cas des applications dilatantes du cercle, à quelques détails techniques près liés à cette écriture de $d$ comme un quotient.

\begin{lm}\label{nucl2}
Soient $k \in \set{1,\dots,q}$ et $m \in G_k$. Il existe des constantes $C >0$ et $0 < \eta <1$ tel que pour tous $v$ et $t$ suffisamment proches de $0$ il existe une suite $\(\phi_n\)_{n \geqslant 0}$ d'éléments de $\B_m$ et une suite $\(l_n\)_{n \geqslant 0}$ d'éléments du dual de $\B_m$ tels que
\[
\forall n \in \N : \|\phi_n\| \|l_n\| \leqslant C \eta^{\sqrt{n}}
\]
et
\[
\L_{m,v,t} = \sum_{n \in \N} l_n \otimes \phi_n
.\]
\end{lm}

\begin{proof}
Il suffit de se rappeler de la preuve de la dépendance analytique de $\L_{m,v,t}$ en $v$ et $t$ (dans \S \ref{22}) : les injections que l'on considère vérifient de telles estimations (par développement en série de Laurent en une coordonnée et en série entière en l'autre).
\end{proof}

\begin{cor}
Soient $k \in \set{1,\dots,q}$ et $m \in G_k$. Il existe des constantes $C_1 >0$ et $C_2 > 0$ telles que pour tous $v$ et $t$ suffisamment proches de $0$ et tout $z \in \C  $ on ait
\[
\log \left| \det\(I- z \L_{m,v,t}\) \right| \leqslant C_1 + C_2 \left| \log |z| \right|^3
.\] 
\end{cor}

\begin{proof}
\'Ecrivons
\[
\det\(I-zL_{m,v,t}\) = \sum_{n=0}^{+ \infty} \alpha_n\(v,t\) z^n
\]
comme dans la preuve de \ref{conv1}, on déduit de \ref{nucl2} que pour tout entier $n$ on a, pour $v$ et $t$ suffisamment proches de $0$ :
\begin{align*}
\left|\alpha_n\(v,t\)\right| & \leqslant C^n \sum_{0 \leqslant i_1 < \dots < i_n} \eta^{\sum_{j=1}^n \sqrt{i_j}} n^{\frac{n}{2}} \leqslant C^n D_1 \exp\(-D_2 n^{\frac{3}{2}}\) n^{\frac{n}{2}}\\
                             & \leqslant D_3 \exp\(-D_4 n^{\frac{3}{2}}\)
\end{align*}
pour certaines constantes $D_3 > 0$ et $D_4 > 0$, où on a utilisé
\[
\sum_{0 \leqslant i_1 < i_2 < \dots <i_n} \eta^{\sum_{j=1}^n \sqrt{i_j}} \leqslant D_1 \exp\(-D_2 n^{\frac{3}{2}}\) 
\]
qui est démontré dans \cite{fried}.

Maintenant si $z \in \C$ avec $|z| \geqslant 1$ on pose 
\[
n_0 = \left\lceil \(\frac{\log\(2 |z| \)}{D_4}\)^2 \right\rceil
\]
 de telle manière que si $n \geqslant n_0$ on a $|z| \exp\(-D_4 \sqrt{n}\) \leqslant \frac{1}{2}$ et donc
\begin{align*}
\left| \det\(I- z \L_{m,v,t}\) \right| & \leqslant  \sum_{n=0}^{+ \infty} \left| \alpha_n \(v,t\) \right| |z|^n \\
      & \leqslant D_3 \sum_{n=0}^{n_0-1} \exp\(-D_4 n^{\frac{3}{2}}\) |z|^n + D_3 \sum_{n=n_0}^{+ \infty} \frac{1}{2^n} \\
      & \leqslant D_3 \( \sum_{n=0}^{+ \infty} \exp\(-D_4 n^{\frac{3}{2}}\) \)|z|^{n_0} + \frac{D_3}{2^{n_0-1}} \\
      & \leqslant D_3 \( \sum_{n=0}^{+ \infty} \exp\(-D_4 n^{\frac{3}{2}}\) \) \exp\(\log |z| \left\lceil \(\frac{\log\(2 |z| \)}{D_4}\)^2 \right\rceil \) + 2 D_3 \\
      & \leqslant D_5 \exp\( \frac{\(\log |z|\)^3}{D_4^2}\) + 2 D_3
\end{align*}
où $D_5 = D_3 \( \sum_{n=0}^{+ \infty} \exp\(-D_4 n^{\frac{3}{2}}\) \) \exp\( \(\frac{\log 2}{D_4}\)^2 + 1\)$. On en déduit aisément le résultat lorsque $|z| \geqslant 1$. On récupère le résultat lorsque $|z| \leqslant 1$ par compacité.
\end{proof}

\begin{lm}\label{254}
Il existe des constantes $A_1 >0$ et $A_2 > 0$ telles que pour tous $v$ et $t$ suffisamment proches de $0$ et tout $z \in \C  $ on ait
\[
\log \left| d\(z,v,t\) \right| \leqslant A_1 + A_2 \left| \log |z| \right|^3
.\] 
\end{lm}

\begin{proof}
Pour tout $r \leqslant 0$, on note $n_{v,t}\(r\)$ le nombre de zéros de $z \mapsto d\(z,v,t\)$ de module inférieur ou égal à $r$ (comptés avec multiplicité). De même, si $k \in \set{1,\dots,q}$ et $m \in G_k$, on note $n_{m,v,t}\(r\)$ le nombre de zéros de $z \mapsto d\(z,v,t\)$ de module inférieur ou égal à $r$. L'application $z \mapsto d\(z,v,t\)$ étant entière pour tout $r \geqslant 0$ on a
\[
n_{v,t}\(r\) = \sum_{k=1}^q \sum_{m \in G_k} \(-1\)^{k+1} n_{m,v,t}\(r\) \leqslant \sum_{0 \leqslant 2l +1 \leqslant q} \sum_{m \in G_{2l+1}} n_{m,v,t}\(r\) 
.\]

On déduit alors de la formule de Jensen qu'il existe des constantes $C_3 >0$ et $C_4 > 0$ telles que pour tout $r \geqslant 0$ on ait 
\[
\int_0^r \frac{n_{v,t}\(s\)}{s} \mathrm{d}s \leqslant \sum_{0 \leqslant 2l +1 \leqslant q} \sum_{m \in G_{2l+1}} \int_0^r \frac{n_{m,v,t}\(s\)}{s} \mathrm{d}s \leqslant C_3 + C_4 \left|\log r\right|^3
.\]

On a alors pour tout $r \geqslant 0$ :
\begin{align*}
n_{v,t}\(r\) & \leqslant 2 \int_r^{2r} \frac{n_{v,t}\(s\)}{s} \mathrm{d}s  \leqslant 2 \int_0^{2r} \frac{n_{v,t}\(s\)}{s} \mathrm{d}s \leqslant C_5 + C_6 \left| \log r \right|^3 
\end{align*}
pour certaines constantes $C_5 >0$ et $C_6 >0$. On a alors pour $r \geqslant 1$ :
\[
\int_r^{+ \infty} \frac{n_{v,t}\(s\)}{s^2} \mathrm{d}s \leqslant \frac{C_5}{r} + C_6 \int_r^{+ \infty}  \frac{\(\log s\)^3}{s^2} \mathrm{d}s
.\]
Et par le changement de variable \og $u= \log s$ \fg{} puis trois intégrations par parties on obtient 
\[
\int_r^{+ \infty}  \frac{\(\log s\)^3}{s^2} \mathrm{d}s = \int_{\log r}^{+ \infty} u^3 e^{-u} \mathrm{d} u = \frac{\(\log r\)^3}{r} + \frac{3\(\log r\)^2}{r} + \frac{6 \log r}{r} + \frac{6}{r}
\] 
et il existe donc des constantes $C_7 > 0$ et $C_8 >0$ telles que pour tout $r \geqslant 1$ on ait
\[
\int_0^r \frac{n_{v,t}\(s\)}{s} \mathrm{d}s + r \int_r^{+ \infty} \frac{n_{v,t}\(s\)}{s^2} \mathrm{d}s \leqslant C_7 + C_8 \(\log r\)^3
.\]
Or $z \mapsto d\(z,v,t\)$ est de genre $0$ (par le corollaire 4 page 18 de la deuxième partie de \cite{Groth}) et $d\(0,v,t\) = 1$ et donc, par le lemme 3.5.1 de \cite{Boas}, pour tout $z \in \C$ tel que $|z| \geqslant 1$ on a
\[
\log |d\(z,v,t\)| \leqslant C_7 + C_8 \(\log |z|\)^3
.\]
Et on récupère le résultat pour $|z| \leqslant 1$ par compacité.

\end{proof}

\begin{prop}\label{vitconv2}
Il existe des constantes $\alpha > 0$ et $\beta > 0$ telles que pour $v$ et $t$ suffisamment proche de $0$ et tout entier naturel $n$ on ait 
\[
\left| a_n\(v,t\) \right| \leqslant \alpha e^{- \beta n^\frac{3}{2}}
,\]
où les $a_n$ sont définis par \eqref{develo}. De plus, les  dérivées partielles des $a_n$ vérifient des estimations similaires  
\end{prop}

\begin{proof}
Les inégalités de Cauchy et le lemme \ref{254} donnent pour tous $r \leqslant 1$ et $n \in \N$ : 
\[
\log \left| a_n\(v,t\) \right|  \leqslant A_1 + A_2 \( \log r \)^3 - n \log r
\]
et on obtient le résultat annoncé en prenant $r = \exp\(\sqrt{\frac{n}{3A_2}}\)$ (quitte à traiter séparément les premiers termes en prenant par exemple $r=1$).

Comme dans le cas des applications dilatantes du cercle, ce résultat reste vrai pour les dérivées des $a_n$ grâce à la formule de Cauchy.
\end{proof}

On a donc une estimation de la vitesse de convergence des séries apparaissant dans la formule donnée par le corollaire \ref{formula} (qui est toujours valable dans ce cadre par la remarque \ref{adaptation}).

\appendix

\section*{Annexe}

\section{Preuve du lemme \ref{pf}}\label{a1}

Soit $K'$ une partie compacte de $V$. La suite d'applications $\(x,z \mapsto \psi_x^n \(z\) \)_{n \in \N}$ étant uniformément bornée sur $K' \times \bar{U}$, on peut en extraire une sous-suite $\(x,z \mapsto \psi_x^{n_k} \(z\) \)_{k \in \N}$ convergeant uniformément sur ce même ensemble. Notons $x,z \mapsto \psi_x^* \(z\)$ la limite, qui est holomorphe sur l'intérieur de $K' \times \bar{U}$. 

Si $z = \(z_1, \dots, z_m \) \in \C^m$ on note $|z| = \sum_{i=1}^d |z_i|$.

Soit $x \in \bul{K'} $, en remarquant que pour tout entier $k$ non-nul on a 
\[
\sup_{D} |\psi_x^{n_k}| \leqslant \sup_{K} |\psi_x^{n_k -1}| \leqslant \sup_{K} |\psi_x^{n_{k -1}}| \leqslant \sup_{D} |\psi_x^{n_{k -1}}|
\]
d'où par passage à la limite
\[
\sup_{z \in D} |\psi_x^*| = \sup_{z \in K} |\psi_x^*|
\]
et donc $\psi_x^*$ est constante par le principe du maximum. On note $z_x$ son unique valeur. La suite de compacts non-vide $\(\psi_x^n\(\bar{D}\)\)_{n \in \N}$ est décroissante donc non-vide et si $z'$ en est un élément alors en notant pour tout $k$ entier naturel $z'_k$ un élément de $\bar{U}$ tel que $\psi_x^{n_k} \(z_k\) = z'$ on a 
\[
|z' - z_x| \leqslant |\psi_x^{n_k} \(z_k\) - \psi_x^*\(z_k\) | \leqslant \sup_D |\psi_x^{n_k} - \psi_x^*|
\] 
et en laissant tendre $k$ vers $+ \infty$ il vient $z_x = z'$ d'où
\[
\bigcap_{n \in \N} \psi_x^n\(\bar{D}\) = \set{z_x}
.\]
Cet ensemble étant stable par $\psi_x$ (puisque $\psi_x \(\bar{D}\) \subset \bar{D}$), il vient que $z_x$ est l'unique point fixe de $\psi_x$ dans $\bar{D}$ et donc l'unique point fixe de $\psi_x$.

$x \mapsto z_x$ sur $\bul{K'}$ est holomorphe comme composées d'applications holomorphes. $V$ étant dénombrable à l'infini on en déduit que $x \mapsto z_x$ est holomorphe sur $V$.

Enfin, la convergence uniforme implique que
\[
\(d\(\psi_x\)_{z_x}\)^{n_k} = d\( \psi_x^{n_k} \)_{z_x} \underset{ k \to + \infty}{\to} d\(\psi_x^*\)_{z_x} = 0
\]
et donc les valeurs propres de $d\(\psi_x\)_{z_x}$ sont de module strictement plus petit que $1$.

\section{Preuve de la formule \texorpdfstring{\eqref{rec}}{e}}\label{a2}
Si $i_0, \dots, i_{n+1} \in I$ sont tels que $a\(i_0,i_1\) = \dots = a\(i_n,i_{n+1}\) = 1$ , $z = \(z_1,z_2\)$ est dans un voisinage de $\bar{U_{i_{n+1}}}$ et $w = \(w_1,w_2\)$ est dans un voisinage de $\bar{V_{i_0}}$ alors 
\begin{align*}
& \frac{1}{2i \pi} \s{\partial D_1^{i_n}}{\( \frac{1}{2i \pi} \s{\partial D_2^{i_n}}{G^{i_n i_{n+1}} \(z, x\) G^{i_0 \dots i_n}\(x,w\)}{x_2}\) }{x_1}  \\
 & = \frac{s_{i_0 \dots i_{n+1}}}{2i \pi} \s{\partial D_1^{i_n}}{\( \frac{1}{2i \pi} \s{\partial D_2^{i_n}}{ \frac{h^{i_n i_{n+1}} \(z,x\) h^{i_0 \dots i_n}\(x,w\)}{x_1 - \phi_u^{i_0 \dots i_n} \(w_1,x_2\)} \frac{\partial_2 \phi_s^{i_0 \dots i_n} \(w_1,x_2\)}{w_2 - \phi_s^{i_0 \dots i_n} \(w_1,x_2\)} \right. \\ & \quad \left. \frac{1}{w_1 - \phi_u^{i_n i_{n+1}} \(x_1,z_2\)} \frac{\partial_2 \phi_s^{i_n i_{n+1}}\(x_1,z_2\)}{x_2 - \phi_s^{i_n i_{n+1}} \(x_1,z_2\)} }{x_2}\) }{x_1} \\ 
 & = \frac{s_{i_0 \dots i_{n+1}}}{2i \pi} \s{\partial D_1^{i_n}}{ \( \frac{h^{i_n i_{n+1}} \(z,\(x_1,\phi_u^{i_n i_{n+1}} \(x_1,z_2\)\)\) h^{i_0 \dots i_n}\(\(x_1,\phi_u^{i_n i_{n+1}} \(x_1,z_2\)\),w\)}{x_1 - \phi_u^{i_0 \dots i_n} \(w_1,\phi_u^{i_n i_{n+1}} \(x_1,z_2\)\)} \right. \\ & \quad \left. \frac{\partial_2 \phi_s^{i_0 \dots i_n} \(w_1,\phi_u^{i_n i_{n+1}} \(x_1,z_2\)\)}{w_2 - \phi_s^{i_0 \dots i_n} \(w_1,\phi_u^{i_n i_{n+1}} \(x_1,z_2\)\)}  \frac{\partial_2 \phi_s^{i_n i_{n+1}}\(x_1,z_2\)}{w_1 - \phi_u^{i_n i_{n+1}} \(x_1,z_2\)}  \) }{x_1} \\
 & = s_{i_0 \dots i_{n+1}} \frac{ h^{i_n i_{n+1}}\(z, \xi^{i_0 \dots i_{n+1}}\(w_1,z_2\)\) h^{i_0 \dots i_n} \(\xi^{i_0 \dots i_{n+1}}\(w_1,z_2\), w\)}{{1- \partial_2 \phi_u^{i_0 \dots i_n}\(w_1,\xi_2^{i_0 \dots i_{n+1}} \(w_1,z_2\)\) \partial_1 \phi_s^{i_n i_{n+1}} \(\xi_1^{i_0 \dots i_{n+1}}\(w_1,z_2\),z_2\) }} \\ & \quad \frac{\partial_2 \phi_s^{i_0 \dots i_n} \(w_1, \xi_2^{i_0 \dots i_{n+1}} \(w_1,z_2\)\)}{w_2 - \phi_s^{i_0 \dots i_n}\(w_1, \xi_2^{i_0 \dots i_{n+1}}\(w_1,z_2\)\)} \frac{\partial_2 \phi_s^{i_n i_{n+1}} \( \xi_1^{i_0 \dots i_{n+1}}\(w_1,z_2\), z_2\)}{z_1 - \phi_u^{i_n i_{n+1}}\( \xi_1^{i_0 \dots i_{n+1}}\(w_1,z_2\),z_2\)} \\
 & = s_{i_0 \dots i_{n+1}} \frac{h^{i_n i_{n+1}}\(z, \xi^{i_0 \dots i_{n+1}}\(w_1,z_2\)\) h^{i_0 \dots i_n} \(\xi^{i_0 \dots i_{n+1}}\(w_1,z_2\), w\)}{1- \partial_2 \phi_u^{i_0 \dots i_n}\(w_1,\xi_2^{i_0 \dots i_{n+1}} \(w_1,z_2\)\) \partial_1 \phi_s^{i_n i_{n+1}} \(\xi_1^{i_0 \dots i_{n+1}}\(w_1,z_2\),z_2\) } \\ & \quad \frac{\partial_2 \phi_s^{i_0 \dots i_n} \(w_1, \xi_2^{i_0 \dots i_{n+1}} \(w_1,z_2\)\)}{w_2 - \phi_s^{i_0 \dots i_n i_{n+1}}\(w_1, z_2\)} \frac{\partial_2 \phi_s^{i_n i_{n+1}} \( \xi_1^{i_0 \dots i_{n+1}}\(w_1,z_2\), z_2\)}{z_1 - \phi_u^{i_0 \dots i_{n+1}}\( w_1,z_2\)} \\
\end{align*}
où on a utilisé la formule de Cauchy puis la formule des résidus (le pôle est bien simple par le lemme \ref{pf}) et la définition de $ \phi_u^{i_0 \dots i_{n+1}}$ et de $\phi_s^{i_0 \dots i_{n+1}}$. Par ailleurs on note que
\begin{align*}
& h^{i_n i_{n+1}}\(z, \xi^{i_0 \dots i_{n+1}}\(w_1,z_2\)\) h^{i_0 \dots i_n} \(\xi^{i_0 \dots i_{n+1}}\(w_1,z_2\), w\) \\
 & = h \( f^{i_n i_{n+1}} \( \xi_1^{i_0 \dots i_{n+1}}\(w_1,z_2\), \phi_s^{i_n i_{n+1}} \( \xi_1^{i_0 \dots i_{n+1}}\(w_1,z_2\) , z_2 \)\)\) \\ & \quad \prod_{k=1}^n h \(f^{i_0 \dots i_k} \(w_1, \phi_s^{i_0 \dots i_n} \(w_1,\xi_2^{i_0 \dots i_{n+1}}\(w_1,z_2\)\)\)\) \\
 & = h \( \phi_u^{i_n i_{n+1}} \( \xi_1^{i_0 \dots i_{n+1}}\(w_1,z_2\), z_2\) , z_2 \) \prod_{k=1}^n h \(f^{i_0 \dots i_k} \(w_1, \phi_s^{i_0 \dots i_{n+1}} \(w_1,z_2\)\)\) \\
 & = h \( \phi_u^{i_0 \dots i_{n+1}} \(w_1,z_2\), z_2\) \prod_{k=1}^n h \(f^{i_0 \dots i_k} \(w_1, \phi_s^{i_0 \dots i_{n+1}} \(w_1,z_2\)\)\) \\
 & = h \( f^{i_0 \dots i_{n+1}} \(w_1, \phi_u^{i_0 \dots i_{n+1}}\(w_1,z_2\)\)\) \prod_{k=1}^n h \(f^{i_0 \dots i_k} \(w_1, \phi_s^{i_0 \dots i_{n+1}} \(w_1,z_2\)\)\) \\
 & = h^{i_0 \dots i_{n+1}} \(w,z\)
.\end{align*}
En dérivant
\[
\xi_2^{i_0 \dots i_{n+1}} \(w_1,z_2\) = \phi_s^{i_n i_{n+1}} \( \phi_u^{i_0 \dots i_n} \(w_1, \xi_2^{i_0 \dots i_{n+1}} \(w_1,z_2\)\), z_2\)
\]
on obtient
\[
\drond{\xi_2^{i_0 \dots i_{n+1}}}{z_2} \(w_1,z_2\) = \frac{\partial_2 \phi_s^{i_n i_{n+1}} \(\xi_1^{i_0 \dots i_{n+1}}\(w_1,z_2\), z_2\)}{1- \partial_2 \phi_u^{i_0 \dots i_n}\(w_1,\xi_2^{i_0 \dots i_n} \(w_1,z_2\)\) \partial_1 \phi_s^{i_n i_{n+1}} \(\xi_1^{i_0 \dots i_{n+1}}\(w_1,z_2\),z_2\)}
\]
ce qui en dérivant \eqref{def2} donne
\begin{equation}\label{signe}
\partial_2 \phi_s^{i_0 \dots i_{n+1}} \(w_1,z_2\) = \frac{\partial_2 \phi_s^{i_0 \dots i_n} \(w_1,\xi_2^{i_0 \dots i_{n+1}} \(w_1,z_2\)\) \partial_2 \phi_s^{i_n i_{n+1}} \(\xi_1^{i_0 \dots i_{n+1}}\(w_1,z_2\), z_2\) }{1- \partial_2 \phi_u^{i_0 \dots i_n}\(w_1,\xi_2^{i_0 \dots i_{n+1}} \(w_1,z_2\)\) \partial_1 \phi_s^{i_n i_{n+1}} \(\xi_1^{i_0 \dots i_{n+1}}\(w_1,z_2\),z_2\)}
\end{equation}
et des calculs précédents on déduit la formule \eqref{rec}.

\section{Preuve de la formule \texorpdfstring{\eqref{tr}}{e}}\label{a3}
Si $i \in I$ et $n_1,n_2$ sont des entiers naturels on note $e_{n_1,n_2}^i$ l'élément de $\B_i$ défini par
\[
\forall \(z_1,z_2\) \in U_i : e_{n_1,n_2}^i \(z_1,z_2\) = z_1^{-n_1 - 1} z_2^{n_2}
\]
dont la norme est $\frac{\(r_2^i\)^{n_2}}{\(r_1^i\)^{n_1+1}} $, et $e_{n_1,n_2}^{i,*}$ l'élément de son dual défini par 
\[
\forall \phi \in \B_i : e_{n_1,n_2}^{i,*} \(\phi\) = \frac{1}{2i \pi} \s{\partial D_1^i}{\( \frac{1}{2i \pi} \s{\partial D_2^i}{w_1^{n_1} w_2^{-n_2-1} \phi\(w_1,w_2\) }{w_2}\)}{w_1}
.\]
On remarquera que $e_{n_1,n_2}^{i,*} \(e_{n_1,n_2}^i\) = 1$ et que la norme de $e_{n_1,n_2}^{i,*}$ est $\frac{\(r_1^i\)^{n_1+1}}{\(r_2^i\)^{n_2}}$.

Si $i_0, \dots, i_n \in I$ sont tels que $a\(i_0,i_1\) = \dots = a\(i_{n-1},i_n\) = 1$ alors pour tout $w$ dans $\bar{V_{i_0}}$, l'application $z \mapsto G^{i_0 \dots i_n} \(z,w\)$ est holomorphe sur un voisinage de $\bar{U_{i_n}}$ et on peut donc écrire
\begin{equation}\label{somme}
G^{i_0 \dots i_n} \(z,w\) = \sum_{n_1,n_2 \in \N} z_1^{-n_1-1} z_2^{n_2} \Theta_{n_1,n_2}^{i_0 \dots i_n} \(w\)
\end{equation}
où
\[
\Theta_{n_1,n_2}^{i_0 \dots i_n} \(w \) = \frac{1}{2i \pi} \s{\partial D_1^{i_n}}{\( \frac{1}{2i \pi} \s{\partial D_2^{i_n}}{x_1^{n_1} x_2^{-n_2-1} G^{i_0 \dots i_n}\(x,w\)  }{x_2}\)}{x_1}
.\]
La formule des résidus donne alors
\[
\Theta_{n_1,n_2}^{i_0 \dots i_n} \(w \) = \frac{s_{i_0 \dots i_n}}{2i \pi} \s{\partial D_2^{i_n}}{ h^{i_0 \dots i_n} \( \( \phi_u^{i_0 \dots i_n} \(w_1,x_2\),x_2\),w\) \frac{\partial_2 \phi_s^{i_0 \dots i_n} \(w_1,x_2\) \(\phi_u^{i_0 \dots i_n} \(w_1,x_2\)\)^{n_1}}{\(w_2 - \phi_s^{i_0 \dots i_n} \(w_1,x_2\) \)x_2^{n_2+1}}}{x_2}
.\]
Comme $\phi_s^{i_0 \dots i_n}$ envoie $\bar{D_1^{i_0}} \times \bar{D_2^{i_n}}$ dans $\bul{D_2^{i_0}}$, pour $\epsilon$ suffisamment petit, $\phi_s^{i_0 \dots i_n}$ envoie $\bar{D_1^{i_0}} \times \bar{D_2^{i_n, \epsilon}}$ dans $\bul{D_2^{i_0}}$, où $D_2^{i_n, \epsilon}$ désigne le disque de centre $0$ et de rayon $\(1+ \epsilon\) r_2^{i_n}$ et on a alors 
\[
\Theta_{n_1,n_2}^{i_0 \dots i_n} \(w \) = \frac{s_{i_0 \dots i_n}}{2i \pi} \s{\partial D_2^{i_n, \epsilon}}{ h^{i_0 \dots i_n} \( \( \phi_u^{i_0 \dots i_n} \(w_1,x_2\),x_2\),w\) \frac{\partial_2 \phi_s^{i_0 \dots i_n} \(w_1,x_2\) \(\phi_u^{i_0 \dots i_n} \(w_1,x_2\)\)^{n_1}}{\(w_2 - \phi_s^{i_0 \dots i_n} \(w_1,x_2\) \)x_2^{n_2+1}}}{x_2}
\] 
puisque $\partial D_2^{i_n} $ et $\partial D_2^{i_n, \epsilon} $ sont homotopes dans l'ensemble de définition de l'intégrande.

Procédant au changement de variable \og $\tilde{w_2} = \phi_s^{i_0 \dots i_n} \(w_1,x_2\)$ \fg{} (on a alors $x_2 = f_2^{i_0 \dots i_n} \(w_1, \tilde{w_2}\)$ et $ \phi_u^{i_0 \dots i_n} \(w_1,x_2\) = f_1^{i_0 \dots i_n} \(w_1,\tilde{w_2}\)$), nous obtenons
\[
\Theta_{n_1,n_2}^{i_0 \dots i_n} \(w \) = \frac{s_{i_0 \dots i_n}}{2i \pi} \s{\Gamma\(w_1\)}{h^{i_0 \dots i_n}\(f^{i_0 \dots i_n}\(w_1,\tilde{w_2}\), w\) \frac{\(f_1^{i_0 \dots i_n}\(w_1, \tilde{w_2}\)\)^{n_1}}{\(w_2 - \tilde{w_2}\) \(f_2\(w_1,\tilde{w_2}\)\)^{n_2+1}} }{\tilde{w_2}}
\]
où $\Gamma\(w_1\)$ est l'image de $\partial D_2^{i_n, \epsilon}$ par $x_2 \mapsto \phi_s^{i_0 \dots i_n} \(w_1,x_2\)$ .

$f_1^{i_0 \dots i_n}$ envoie $\bar{D_1^{i_0}} \times \phi_s^{i_0 \dots i_n} \( \bar{D_1^{i_0}} \times \bar{D_2^{i_n}}\)$ dans $\bul{D_1^{i_n}}$ et donc $f_1^{i_0 \dots i_n}$ envoie $\bar{D_1^{i_0}} \times \phi_s^{i_0 \dots i_n} \( \bar{D_1^{i_0}} \times \bar{D_2^{i_n, \epsilon}}\)$ dans $\bul{D_1^{i_n}}$ pour $\epsilon$ assez petit, on note alors
\[
\theta = \frac{1}{r_1^{i_n}} \sup_{\bar{D_1^{i_0}} \times \phi_s^{i_0 \dots i_n} \( \bar{D_1^{i_0}} \times \bar{D_2^{i_n, \epsilon}}\)} |f_1| < 1
.\]

En notant $\delta > 0$ la distance entre les compacts $\phi_s^{i_0 \dots i_n} \( \bar{D_1^{i_0}} \times \bar{D_2^{i_n, \epsilon}}\)$ et $\hat{\C} \setminus \bul{D_2^{i_n}}$ et $M$ une majoration de $h$ sur les rectangles considérées, on obtient la majoration
\begin{align*}
\sup_{\bar{V_{i_0}}} \left| \Theta_{n_1,n_2}^{i_0 \dots i_n} \right| & \leqslant \frac{\left| \Gamma\(w_1\) \right|}{2 \pi} M^n \frac{\(\theta r_1^{i_n}\)^{n_1}}{\delta \( \(1+ \epsilon\) r_2^{i_n} \)^{n_2+1}} \leqslant C \frac{\(\theta r_1^{i_n}\)^{n_1}}{\( \(1+ \epsilon\) r_2^{i_n} \)^{n_2+1}}
\end{align*}
où la constante $C$ ne dépend pas de $n_1$ et $n_2$.

On note alors $\Theta_{n_1,n_2}^{i_0 \dots i_n, *}$ la forme linéaire définie sur $\B_{i_0}$ par
\[
\Theta_{n_1,n_2}^{i_0 \dots i_n, *} \( \phi\) = \frac{1}{2i \pi} \s{\partial D_1^{i_0}}{ \( \frac{1}{2i \pi} \s{\partial D_2^{i_0}}{\Theta_{n_1,n_2}^{i_0 \dots i_n} \(w_1,w_2\) \phi\(w\)}{w_2}\)}{w_1} 
\]
et on a l'estimation
\[
\|\Theta_{n_1,n_2}^{i_0 \dots i_n, *}\| \leqslant C'  \frac{\(\theta r_1^{i_n}\)^{n_1}}{\( \(1+ \epsilon\) r_2^{i_n} \)^{n_2+1}}
\]
où la constante $C'$ ne dépend pas de $n_1$ et $n_2$. La série
\[
\sum_{n_1,n_2 \in \N} \Theta_{n_1,n_2}^{i_0 \dots i_n, *} \otimes e_{n_1,n_2}^{i_n}
\]
converge donc vers un noyau $\frac{2}{3}$-sommable. La même estimation nous permet d'intervertir l'intégrale de la définition \eqref{def3} et la somme de la formule \eqref{somme} et prouve que ce noyau représente $ \L^{i_0 \dots i_n} $ qui est donc $\frac{2}{3}$-sommable.

Dans le cas particulier où $i_0 = i_n$, l'opérateur $\L^{i_0 \dots i_n}$ est un endomorphisme $\frac{2}{3}$-sommable de $\B_{i_0}$ et 
\begin{align*}
\tr{\L^{i_0 \dots i_n}} & = \sum_{n_1,n_2 \in \N} \Theta_{n_1,n_2}^{i_0 \dots i_n, *} \( e_{n_1,n_2}^{i_n} \) \\
   & = \sum_{n_1,n_2 \in \N} \frac{1}{2i \pi} \s{\partial D_1^{i_0}}{ \( \frac{1}{2i \pi} \s{\partial D_2^{i_0}}{\Theta_{n_1,n_2}^{i_0 \dots i_n} \(w_1,w_2\) w_1^{-n_1-1}w_2^{n_2}}{w_2}\)}{w_1} \\
   & = \frac{s_{i_0 \dots i_n}}{2i \pi} \s{\partial D_1^{i_0}}{ \( \frac{1}{2i \pi} \s{\partial D_2^{i_0}}{\sum_{n_1,n_2 \in \N} \(\frac{1}{2i \pi} \s{\Gamma\(w_1\)}{h^{i_0 \dots i_n}\(f^{i_0 \dots i_n}\(w_1,\tilde{w_2}\), w\)  \right. \right. \\ & \quad \left. \left. \frac{\(f_1^{i_0 \dots i_n}\(w_1, \tilde{w_2}\)\)^{n_1}}{\(w_2 - \tilde{w_2}\) \(f_2\(w_1,\tilde{w_2}\)\)^{n_2+1}} }{\tilde{w_2}}\)w_1^{-n_1-1} w_2^{n_2}}{w_2}\)}{w_1} \\
   & = \frac{s_{i_0 \dots i_n}}{2i \pi} \s{\partial D_1^{i_0}}{ \( \frac{1}{2i \pi} \s{\partial D_2^{i_0}}{ \(\frac{1}{2i \pi} \s{\Gamma\(w_1\)}{h^{i_0 \dots i_n}\(f^{i_0 \dots i_n}\(w_1,\tilde{w_2}\), w\)  \right. \right. \\ & \quad \left. \left. \sum_{n_1,n_2 \in \N} \frac{\(f_1^{i_0 \dots i_n}\(w_1, \tilde{w_2}\) w_1^{-1} \)^{n_1}}{\(w_2 - \tilde{w_2}\) \(f_2\(w_1,\tilde{w_2}\)w_2^{-1} \)^{n_2+1}} }{\tilde{w_2}}\)\frac{1}{w_1 w_2}}{w_2}\)}{w_1} \\
   & = \frac{s_{i_0 \dots i_n}}{2i \pi} \s{\partial D_1^{i_0}}{ \( \frac{1}{2i \pi} \s{\partial D_2^{i_0}}{ \(\frac{1}{2i \pi} \s{\Gamma\(w_1\)}{h^{i_0 \dots i_n}\(f^{i_0 \dots i_n}\(w_1,\tilde{w_2}\), w\) \right. \right. \\ & \quad \left. \left. \frac{1}{w_1 - f_1^{i_0 \dots i_n} \(w_1,\tilde{w_2}\)} \frac{1}{f_2^{i_0 \dots i_n} \(w_1, \tilde{w_2}\)-w_2} \frac{1}{w_2 -\tilde{w_2}}}{\tilde{w_2}}\)}{w_2}\)}{w_1}  \\
   & = \frac{s_{i_0 \dots i_n}}{2i \pi} \s{\partial D_1^{i_0}}{ \( \frac{1}{2i \pi} \s{\Gamma\(w_1\)}{ \(\frac{1}{2i \pi} \s{\partial D_2^{i_0}}{h^{i_0 \dots i_n}\(f^{i_0 \dots i_n}\(w_1,\tilde{w_2}\), w\) \right. \right. \\ & \quad \left. \left. \frac{1}{w_1 - f_1^{i_0 \dots i_n} \(w_1,\tilde{w_2}\)} \frac{1}{f_2^{i_0 \dots i_n} \(w_1, \tilde{w_2}\)-w_2} \frac{1}{w_2 -\tilde{w_2}}}{w_2}\)}{\tilde{w_2}}\)}{w_1} \\
   & = \frac{s_{i_0 \dots i_n}}{2i \pi} \s{\partial D_1^{i_0}}{ \( \frac{1}{2i \pi} \s{\Gamma\(w_1\)}{ h^{i_0 \dots i_n}\(f^{i_0 \dots i_n}\(w_1,\tilde{w_2}\), w\) \right.  \\ & \quad \left.  \frac{1}{w_1 - f_1^{i_0 \dots i_n} \(w_1,\tilde{w_2}\)} \frac{1}{f_2^{i_0 \dots i_n} \(w_1, \tilde{w_2}\)-\tilde{w_2}} }{\tilde{w_2}}\)}{w_1} .
\end{align*}
En notant, $W_2\(w_1\)$ l'unique point fixe de $w_2 \mapsto \phi_s^{i_0 \dots i_n} \(w_1,w_2\)$, qui est bien défini par le lemme \ref{pf}, on remarque que
\[
|\partial_2 f_2^{i_0 \dots i_n} \(w_1,W_2\(w_1\)\)| = \frac{1}{|\partial_2 \phi_s^{i_0 \dots i_n} \(w_2,W_2\(w_1\)\)|} > 1
\]
et donc par la formule des résidus on a
\begin{align*}
\frac{1}{2i \pi} \s{\Gamma\(w_1\)}{h^{i_0 \dots i_n}\(f^{i_0 \dots i_n}\(w_1,\tilde{w_2}\), w\)  \frac{1}{w_1 - f_1^{i_0 \dots i_n} \(w_1,\tilde{w_2}\)} \frac{1}{\tilde{w_2} - f_2^{i_0 \dots i_n} \(w_1, \tilde{w_2}\)} }{\tilde{w_2}} \\ = h^{i_0 \dots i_n}\(f^{i_0 \dots i_n}\(w_1,W_2\(w_1\)\), \(w_1,W_2\(w_1\)\)\)  \frac{1}{w_1 - f_1^{i_0 \dots i_n} \(w_1,W_2\(w_1\)\)} \frac{1}{1 - \partial_2 f_2^{i_0 \dots i_n}\(w_1,W_2\(w_1\)\)} 
.\end{align*}
En remarquant que 
\[
W_2'\(w_1\) \(1 - \partial_2 f_2^{i_0 \dots i_n} \(w_1,W_2\(w_1\)\) \)= \partial_1 f_2^{i_0 \dots i_n} \(w_1,W_2\(w_1\)\)
\]
et en appliquant de nouveau la formule des résidus (on a bien un pôle simple car le point fixe $x_{i_0 \dots i_n}$ est hyperbolique), il vient
\begin{align*}
\tr{\L^{i_0 \dots i_n}} & = - \frac{s_{i_0 \dots i_n} h^{i_0 \dots i_n} \( f^{i_0 \dots i_n} \(x^{i_0 \dots i_n}\), x^{i_0 \dots i_n} \) }{\(1-\partial_2 f_2^{i_0 \dots i_n} \(x^{i_0 \dots i_n}\)\) \( 1 - \partial_1 f_1^{i_0 \dots i_n} \(x^{i_0 \dots i_n}\) - W_2'\(x_1^{i_0 \dots i_n}\) \partial_2 f_1^{i_0 \dots i_n} \(x^{i_0 \dots i_n}\)\)} \\
    & = - \frac{s_{i_0 \dots i_n} \prod_{k=1}^n h\(f^{i_0 \dots i_k} \(x^{i_0 \dots i_n}\)\)}{\det \(df^{i_0 \dots i_n}\(x^{i_0 \dots i_n}\) - I \)}
.\end{align*}

Or, si $n$ est un entier naturel non-nul, la formule \eqref{puiss} donne
\[
\tr \(\L^n\) = \sum_{ a\(i_0,i_1\) = \dots = a\(i_{n-1},i_n\) = 1 } \tr \(b_{i_n} \circ \L^{i_0 \dots i_n} \circ \pi_{i_0}\)
.\]
On remarque que les opérateurs qui apparaissent dans cette somme sont $\frac{2}{3}$-sommable, que si $i_0 \neq i_n$, l'opérateur $b_{i_n} \circ \L^{i_0 \dots i_n} \circ \pi_{i_0}$ n'a pas de spectre en dehors de $0$ et sa trace est donc nulle et si $i_0$ = $i_n$, cet opérateur a même spectre que $\L^{i_0 \dots i_n}$ et donc même trace. D'où la formule \eqref{tr}.

\section*{Remerciements}
Je remercie Hans Henrik Rugh, Damien Thomine et Viviane Baladi pour leurs commentaires pertinents et leurs relectures attentives.

\bibliography{biblio}

\end{document}